\documentclass[11pt]{amsart}
\usepackage[toc,page]{appendix}
\usepackage[utf8x]{inputenc}

\usepackage{amsmath, amsthm, amsfonts, amssymb}
\usepackage{epsfig, enumerate}
\usepackage{color}

\usepackage{hyperref}

\newtheorem{maintheorem}{Theorem} 
\newtheorem{theorem}{Theorem}
\newtheorem{lemma}[theorem]{Lemma}
\newtheorem{coro}[theorem]{Corollary}
\newtheorem{prop}[theorem]{Proposition}

\newtheorem*{assumptions*}{Assumptions}

\newtheorem*{rem*}{Remark}

\theoremstyle{remark}
\newtheorem{remark}[theorem]{Remark}
\newtheorem*{remark*}{Remark}

\theoremstyle{definition}

\newcommand{\W}{{\mathbf W}}

\newcommand{\NN}{{\mathbb N}}

\newcommand{\QQ}{{\mathbb Q}}
\newcommand{\RR}{{\mathbb R}}

\newcommand{\ZZ}{{\mathbb Z}}

\newcommand{\be}[1]{\begin{equation} \label{#1} }
\newcommand{\ee}{\end{equation}}
\newcommand{\beq}{\begin{equation}}

\def \Diff{{\rm Diff}}
\def \al{{\alpha}}

\def \W{{\mathcal W}}

\numberwithin{theorem}{section}
\numberwithin{equation}{section}

\author{Danijela Damjanovi\'c}

\address[Damjanovi\'c]{Department of mathematics, Kungliga Tekniska högskolan, Lindstedtsvägen 25, SE-100 44 Stockholm, Sweden.} 
\email{ddam@kth.se}

\author{Disheng Xu}

\address[Xu]{Department of mathematics, Kungliga Tekniska högskolan, Lindstedtsvägen 25, SE-100 44 Stockholm, Sweden.}
\email{dishengxu1989@gmail.com}


\begin{document}

\title[Cocycle rigidity]{Diffeomorphism group valued cocycles over higher rank abelian Anosov actions}

\date{\today}
\maketitle

\begin{abstract} 

We prove that every smooth diffeomorphism group valued cocycle over certain $\ZZ^k$ Anosov actions on tori (and more generally on infranilmanifolds), is a  smooth coboundary on a finite cover, if the cocycle is center bunched and trivial at a fixed point. For smooth cocycles which are not trivial at a fixed point, we have smooth reduction of cocycles to constant ones, when lifted to the universal cover. These results on cocycle trivialisation apply, via the existing global rigidity results, to maximal Cartan  $\ZZ^k$ ($k\ge 3$) actions by Anosov diffeomorphisms (with at least one transitive), on any compact smooth manifold. This is the first rigidity result for cocycles over $\ZZ^k$ actions  with values in diffeomorphism groups which does not require any restrictions on the smallness of the cocycle, nor on the diffeomorphism group.
\end{abstract}
 
\tableofcontents
\section{Introduction} 
There has been a large body of work studying smooth cohomology over Anosov diffeomorphisms and flows, since the celebrated work of Livsic \cite{Livsic1}, \cite{Livsic2}, on real valued cocycles where vanishing of periodic orbit obstructions was proven sufficient for (smooth) cocycle trivialisation. Cocycles taking values in other groups have been studied extensively since the work of Livsic, markedly Livsic theorem for matrix cocycles was proved in \cite{Kalinin}. Otherwise, for cocycles taking values in more general groups, such as diffeomorphism groups, there are results for cocycles close to identity \cite{NTnonabelian2} or for improving regularity of the solution to a cohomological equation \cite{NT98} \cite{LW}. Recently, there has been work done in the direction of proving  Livsic theorem for non-small  cocycles taking values in diffeomorphism groups of manifolds of small dimension \cite{KP}.

For Anosov actions of larger abelian groups it was discovered in \cite{KatSpa} that certain irreducibility criterion on the action implies that obstructions for trivialisation of real valued cocycles for individual action elements vanish for cocycles over the action. As a consequence it was obtained that the first smooth cohomology over  such actions is almost trivial, that is: it reduces to constant cocycles. This property was labeled \emph{cocycle rigidity} and it was crucial in  proving perturbative rigidity results for such actions. Cocycles over abelian Anosov and partially hyperbolic actions, taking values in compact Lie groups, and small cocycles taking values in more general Lie groups, have been studied extensively as well, and smooth {cocycle rigidity}, or a classification, was obtained in many cases \cite{KNT}, \cite{NTnonabelian1}, \cite{NTnonabelian2}, \cite{DKpc}, \cite{KN}, \cite{BN}.

In this paper we are interested in cocycles over abelian Anosov actions, taking values in the group of smooth diffeomorphisms $\Diff(N)$ of a compact smooth manifold $N$. Our main result is a Livsic type theorem, which is in the same time a rigidity statement, for algebraic Anosov actions on infranilmanifolds, under certain irreducibility assumptions on the action. Namely, we show that any smooth center-bunched cocycle which takes values in $\Diff(N)$, is a smooth coboundary on some finite cover, if it is trivial at some fixed point of the action. Also, similar condition at fixed points of action elements suffices if the action does not have a fixed point. In particular, this result implies vanishing of obvious periodic orbit obstructions for cocycle trivialization for any action element. Equivalently, this means that for the partially hyperbolic extensions built over the given Anosov action via a $\Diff(N)$ valued cocycle we have:  if the extension  pointwise fixes  one fiber, then the extension reduces to a product action. We note that the corresponding  local statement for  $\Diff(N)$ valued cocycles which are \emph{close to the identity}, with the same condition on fixing a fiber, appears in \cite[Theorem 3.1]{NT01}, where it is used for obtaining a local rigidity result for perturbations of certain property (T) group actions. 

As a corollary, due to global rigidity result of Kalinin and Spatzier \cite{KSp}, this kind of rigidity in cohomology holds for any maximal Cartan $\ZZ^k$ ($k\ge 2$) action  on a smooth compact manifold, if all action elements are Anosov and at least one is transitive. We remark that in previous work on diffeomorphism group valued cocycles, either localization hypothesis or all periodic data  was needed,  while here we only need the natural assumption on center-bunching and data on a finite set. This is the first cocycle rigidity result for  $\Diff(N)$ valued cocycles over abelian Anosov  actions, which does not require any restrictions on closeness to identity, or the target diffeomorphism group $\Diff(N)$.

In our approach  we  consider partially hyperbolic extensions via $\Diff(N)$ valued cocycles over Anosov abelian actions. We show that the action-invariant structures (in particular, action-invariant foliations) for the Anosov action lift to the invariant structures for the partially hyperbolic extension and we show their regularity.  The crucial point is proving the existence of a smooth horizontal foliation 
which is uniformly transverse to the fibers $N$, without the smallness assumption on the cocycle. This allows us to use the holonomy map of the horizontal foliation to construct a well defined smooth transfer map on the universal cover from the given cocycle to a constant cocycle. In this case we say the cocycle is \emph{essentially smoothly cohomologous to a constant}. For cocycles which are identity at a fixed point (or at fixed points for action generators), it is essentially ergodicity of the elements of the base Anosov algebraic action, which implies existence of a \emph{finite} cover on which the cocycle is a smooth coboundary.  The main difference between our approach and former work on $\Diff(N)$ valued cocycles close to Id \cite{NT01}, \cite{KNT}, \cite{NTnonabelian1}, is that for cocycles close to Id the extended actions are small perturbations of product actions, which implies  regularity of action-invariant foliations for the extended action. For non-small cocycles the method we use for proving  regularity of these foliations is inspired by our work in  \cite{DX1} on partially hyperbolic actions with compact center foliation. Rather than using leaf conjugacy from \cite{HPS} we use the $C^r$ section theorem in a rather technical way, which can be viewed as an extension of argument in \cite{KS07} to partially hyperbolic case. 

We apply the results in this paper to the classification problem for partially hyperbolic higher rank actions with compact center foliation \cite{DX1}.



\section{Setting and statements}

\color{black}
\subsection{Anosov $\ZZ^k$ actions on infranilmanifolds}

Recall that $f\in \Diff^1(M)$ is called \emph{partially hyperbolic} if there is a $Df$-invariant splitting $TM = E^{s} \oplus E^{c} \oplus E^{u}$ of the tangent bundle of $M$ such that for some $k \geq 1$, any $x \in M$, and any choice of unit vectors $v^{s} \in E^{s}_{x}$, $v^{c} \in E^{c}_{x}$, $v^{u} \in E^{u}_{x}$,
\[
\|Df^{k}(v^{s})\| < 1< \|Df^{k}(v^{u})\|,
\]
\[
\|Df^{k}(v^{s})\| < \|Df^{k}(v^{c})\| < \|Df^{k}(v^{u})\|.
\]
If $E^u$ and $E^s$ are non trivial and $E^c$ is trivial then $f$ is called \emph{Anosov}.

Now we consider a $\ZZ^k-$action $\al$ on a compact manifold $M$ by diffeomorphisms. The action is called \emph{Anosov} if there is an element that acts as an Anosov diffeomorphism. Recall that a compact nilmanifold is the quotient of a simply connected nilpotent Lie group G by a cocompact discrete subgroup $\Gamma$,
and a compact infranilmanifold is a manifold that is finitely covered by a compact nilmanifold. A linear automorphism of a nilmanifold $G/\Gamma$ is a homeomorphism that is the projection of some $\Gamma$-preserving automorphism of $G$. An affne automorphism of $G/\Gamma$ is the composition of a linear automorphism of $G/\Gamma$ and a left translation. An affne automorphism of a compact infranilmanifold is a homeomorphism that lifts to an affne nilmanifold automorphism on a finite cover. All currently known examples of Anosov diffeomorphisms are topologically conjugated to affine automorphisms of infranilmanifolds.

\subsection{Lyapunov distributions and irreducibility conditions}Suppose $\mu$ is an ergodic probability measure for an Anosov $\ZZ^k$ action $\al$ on a compact manifold $M$. By commutativity, the Lyapunov decompositions for individual elements (cf.\cite{Ose}) of $\ZZ^k$ can be refined to a joint $\al-$invariant splitting. By multiplicative ergodic theorem \cite{Ose} there are finitely many linear functionals $\chi$ on $\ZZ^k$, a $\mu$ full measure set $P$, and an $\al$-invariant measurable splitting of the tangent bundle $TM =\oplus E_\chi$
over $P$ such that for all $a\in\ZZ^k$
and $v\in E_\chi$, the
Lyapunov exponent of $v$ is $\chi(a)$,
The splitting $\oplus E_\chi$ is called the \emph{Lyapunov decomposition}, and the linear functionals
$\chi$ are called the \emph{Lyapunov functionals of $\al$}. The hyperplanes $\ker_\chi\subset \RR^k$ are called
the \emph{Lyapunov hyperplanes}, and the connected components of $\RR^k-\cup_\chi \ker_\chi$are called the \emph{Weyl chambers} of $\al$. The elements in the union of the Weyl chambers are called \textit{regular}. 

For any Lyapunov functional $\chi$ the \textit{coarse Lyapunov distribution} is the direct sum of all Lyapunov spaces with Lyapunov functionals positively proportional to $\chi$: $E^\chi:=\oplus E_{\chi'}, \chi'=c\chi, c>0$. In the presence of sufficiently many Anosov elements (an Anosov element in each Weyl chamber) and if the invariant measure is of full support (such a measure always exists if there is a transitive Anosov element in the action) the coarse Lyapunov distributions are intersections of stable distributions for various elements of the action, they are well defined everywhere, H\"older continuous, and tangent to foliations with smooth leaves.  (For more details see Section 2 in \cite{KS} or Section 2.2 in \cite{KSp}). Moreover, for any other action invariant measure of full support, and Anosov elements in each Weyl chamber, the coarse Lyapunov distributions will be the same, as well as the Weyl chamber picture.

The following properties of $\ZZ^k$ actions have been used in a large body of work to describe irreducibility of the action and they will appear in the main theorems of this paper:

- $\al$ is \emph{maximal} if there are exactly $k+1$ coarse Lyapunov exponents which correspond to $k+1$ distinct Lyapunov hyperplanes, and if Lyapunov hyperspaces are in general position (namely, if no Lyapunov hyperspace contains a non-trivial intersection of two other Lyapunov hyperspaces). 
 
- $\al$ is  \emph{totally non-symlpectic} (TNS) if there are no negatively proportional Lyapunov exponents. 

- $\al$  is  \emph{Cartan} if all coarse Lyapunov distributions are one-dimensional.

-  $\al$  is  \emph{resonance-free} with respect to  an invariant ergodic measure $\mu$ if for any Lyapunov functionals $\chi_i$ , $\chi_j$, and $\chi_l$ such that $\chi_i$ is not positive propositional to $\chi_j$, the functional $(\chi_i − \chi_j )$ is not proportional to $\chi_l$. \footnote{The resonance free assumption here is slightly weaker than that in \cite{KS07}.}

- $\al$ is \emph{full} if for every coarse Lyapunov distribution $E_i$, there exists a regular element $a$ such that  $E_i=E^u_a$, and 
 $\al$ has at least two distinct Lyapunov hyperplanes. 
 
 Classical examples of maximal Cartan actions are $\mathbb Z^k$ actions on the torus $\mathbb T^{k+1}$ by toral automorphisms. Maximality implies a special property of Weyl chambers: namely that there is any combination of signatures of Lyapunov functionals among the Weyl chambers, except all positive, and all negative. In particular, for any Lyapunov functional there is a Weyl chamber in which that Lyapunov functional is positive and all others are negative. This is what we labeled a \textit{full} action. It is easy to see that maximality in general implies the action is full. If $\mathbb Z^k$ action has $r$ distinct Lyapunov hyperplanes and is full, in case when $r\ge k$ and the planes are in general position, by counting the Weyl chambers it is easily checked that the action must be maximal, that is $r= k+1$.   Will show in Lemma \ref{properties} that that fullness implies TNS and resonance free. One can construct examples of actions which are full but are not maximal by taking products of maximal Cartan actions for example. One can also construct examples which are full and are not maximal by "complexifying" maximal Cartan actions (see \cite{DX1} for a concrete construction on $\mathbb T^6$).  We remark here that all the properties listed above are properties of the Weyl chambers structure of the action, and therefore do not depend on the invariant measure, except for the resonance free property.

%


\subsection{Regularity}\label{subsection regl}
Suppose $M,N$ are smooth  compact manifolds, a map $f: M\to N$ is called $C^r, r\notin \ZZ, r> 1$ if $f$ is $C^{[r]}$ and the $[r]$th order partial derivatives of $f$ is uniformly H\"older continuous with exponent $r-[r]$. We denote by $\Diff^r(N)$ the group of $C^r$ diffeomorphisms on $N$. A map $h:M\to \Diff^s(N)$ is called $C^r, 1\leq r\leq s$ if $h(\cdot)(\cdot):M\times N\to N$ is a $C^r-$map. 

In this paper we will study the regularity for many objects, for example, foliations, diffeomorphisms, coboundaries, cocycles, etc. An object is called $C^{s+}$ if it is $C^{s+\epsilon}$ for some 
$\epsilon>0$, and $C^{s-}$ if it is $C^{s-\epsilon}$ for any small $\epsilon$. A family of objects are called uniformly $C^{s+}$ if they are uniformly $C^{s+\epsilon}$ for some $\epsilon>0$, and uniformly $C^{s-}$ if they are uniformly $C^{s-\epsilon}$ for $\epsilon$ arbitrary small.

\subsection{Cocycles with values in diffeomorphism groups}\label{subsedction: def ess coho}Suppose $M, N$ are smooth manifolds. Let $\al$ be a $\ZZ^k-$action on $M$. In this section we assume $1\leq r\leq s\leq \infty$. A map $\beta : \ZZ^k\times M\to \Diff^s(N)$ is called a \textit{cocycle} (with values in group $\Diff^s(N)$) over $\al$ if it satisfies:
$$\beta(a + b, x) = \beta(a, \al(b)\cdot x)\beta(b, x), a, b\in \ZZ^k, x\in M.$$
In addition $\beta$ is called $C^r$ if for any $a\in \ZZ^k$, $\beta(a,\cdot): M\to \Diff^s(N), $ is a $C^r$ map. And $\beta$ is a \emph{constant cocycle} if $\beta$ does not depend on the second coordinate. 

For any $C^s-$cocycle $\beta$, we say $\beta$ is \textit{$C^r -$cohomologous to constant} if there is a homomorphism (constant cocycle)
$\beta_0:\ZZ^k\to \Diff^s(N) $ and a $C^r$ map $h:M\to \Diff^s(N)$ such that
$$\beta(a, x) = h(\al(a)\cdot x)\beta_0(a)h(x),x\in M, a\in\ZZ^k $$ 
A $C^s$ cocycle $\beta$ is a $C^r-$\emph{coboundary} if it is $C^r-$cohomologous to the trivial cocycle. We say $\beta$ is \textit{essentially $C^r-$cohomologous to constant} if there is a cover $(p,\hat{M})$ of $M$  such that the lifted cocycle $$\hat{\beta}(\cdot, \cdot):=\beta(\cdot, p(\cdot)):\ZZ^k\times \hat{M}\to \Diff^s(N)$$
over a lift of $\al$, \color{black}  is  $C^r-$cohomologous to constant.

We call $\beta$ trivial at a point $x$ if $\beta(a,x)$ is the identity map in  $\Diff^s(N)$ for any $a\in\ZZ^k$. We will call a cocycle $\beta$ \emph{fixed point trivial} if there exists a set $S$ of generators of $\ZZ^k$,  for any $a\in S$ there is a fixed point $x_a$ of $\al(a)$ such that $\beta(a, x_a)=id$. In particular, if $\al$ has a fixed point $x_0$, then $\beta$ is fixed point trivial if $\beta$ is trivial at $x_0$.
\begin{remark}\label{rema: ano dif has fix pnt}In \cite{Sm}, Smale conjectured that all Anosov diffeomorphisms on connected compact
manifolds have fixed points, this assertion holds on any infranilmanifold, cf. \cite{M},\cite{Ma},\cite{F},\cite{So}, therefore for actions $\al$ which we consider in this paper, every regular element of the action has a fixed point. 
\end{remark}
\begin{remark}\label{rema: comm fix pnt}It is quite common that $\al$ has a fixed point. In fact there is always a subgroup $A\subset \ZZ^k$ of finite index such that $\al|_A$ has a fixed point $x_0\in M$, this property will be used in the subsequent proofs.\footnote{For any regular $a\in \ZZ^k$, for any $p>0$, the set of $p-$periodic points for $\al(a)$ is clearly discrete hence finite. Moreover it is  $\al-$invariant. Then for any $p$ such that $\al(a)$ has non-empty set of $p-$periodic point set $F_{pa}:=\mathrm{Fix(\al(pa))}$, the restriction of $\al$ on $A:=(\#F_{pa})!\ZZ^k$ has a common fixed point $x$.}  
\end{remark}

\subsection{Bunching conditions}\label{subsection: bunching}Suppose $f$ is an Anosov diffeomorphism on a compact manifold $M$ and $E^s_f, E^u_f$ are the stable and unstable bundles of $f$ respectively. Then we have a $\ZZ-$action $\al_f:\ZZ\to \Diff^1(M)$ on $M$ such that $\al_f(n,x)=f^n(x),x\in M$. Consider a cocycle $\beta:\ZZ\times M\to \Diff^1(N)$ over $\al_f$ where $N$ is a compact manifold, $\beta$ is called $r-$bunched for some $r\geq 0$ if there exists $k\geq 1$ such that,
\begin{eqnarray}\label{eqn: bunching condition}
\sup_{x\in M}\|D_xf^k|_{E^u_f}^{-1}\|\cdot \|D\beta(k,x)\| &<&1\\\label{eqn: bunching on center}
\sup_{x\in M}\|D_xf^k|_{E^u_f}^{-1}\|\cdot\| D\beta(k,x) \|\cdot  \|D\beta(k,x)^{-1}\|^r &<&1
\end{eqnarray}
We say that $\beta$ is \emph{center-bunched} if it is $1-$bunched, and \emph{$\infty$-bunched} if it is $r$-bunched for every $r \geq 1$. In particular, if $\dim N=1$ and $\beta$ is $0-$bunched then $\beta$ is center-bunched.
\begin{remark}
When $r=1$, our bunching condition is similar to \emph{$\lambda-$center bunching} assumption in \cite{KN} for actions.  There is a similar bunching condition in the study of single partially hyperbolic diffeomorphism, in particular the case $r = 1$ corresponds to the \emph{center-bunching} condition considered by Burns and Wilkinson in their proof of the ergodicity of accessible, volume-preserving, center-bunched $C^{2}$ partially hyperbolic diffeomorphisms \cite{BW10}. 
\end{remark}



For a cocycle $\beta$ over higher rank abelian Anosov action $\al$ on $M$, $\beta:\ZZ^k \times M\to \Diff^1(N)$, we say $\beta$ is $r-$bunched for some $r\in [0,\infty]$ if in every Weyl chamber of the action $\al$ there is an element $a\in \ZZ^k$ such that the cocycle $\beta:\ZZ \times M\to \Diff^1(N)$ over the $\ZZ-$action generated by $\al(a)$ is $r-$bunched in the sense above.

\subsection{Statements of the main results}Suppose  $M,N$ are compact connected smooth manifolds and $\al$ is a smooth Anosov action of $\ZZ^k$ on $M$. The following are the main results of this paper. Suppose $\beta: \ZZ^k\times M\to \Diff^{\infty}(N) $ is an $C^{\infty}-$cocycle over $\al$ and $r\geq 1$, then we have the following result for actions which are full, which is a condition independent on the invariant measure.

\begin{maintheorem}\label{thm: cond A}
If $\al$ is full and $M$ is an infranilmanifold, then 
\begin{enumerate}
\item $\beta$ is essentially $C^{r}-$cohomologous to constant if $\beta$ is $r-$bunched.
\item There is a finite cover of $M$ (which only depends on $\al$) such that if $\beta$ is center-bunched then $\beta$ is   fixed point trivial if and only if  $\beta$ lifts to a $C^{\infty}-$coboundary.
\end{enumerate} 
\end{maintheorem}

Due to the existence of a global rigidity result for maximal Cartan actions \cite{KSp}, we obtain the following corollary:

\begin{maintheorem}\label{thm: Cartan}If $\al$ is a maximal Cartan $\ZZ^k$, $k\ge 3$, action on a smooth compact manifold $M$, with all non-trivial elements Anosov, and at least one element transitive, then (1),(2) in Theorem \ref{thm: cond A} hold.
\end{maintheorem}

Theorem \ref{thm: cond A} and \ref{thm: Cartan} are the special cases of the following more general result. Suppose $\beta: \ZZ^k\times M\to \Diff^{s}(N) $ is a $C^{s}-$cocycle over $\al$.

\begin{maintheorem}\label{thm: main}Suppose $M$ is an infranilmanifold. If there exists an $\al-$invariant ergodic measure $\mu$ with full support such that $\al$ is TNS and resonance free with respect to $\mu$, then
\begin{enumerate}
\item  $\beta$ is essentially $C^{r}-$cohomologous to constant if $\beta$ is $r-$bunched and $r\geq 1, s>r+1$.
\item There is a finite cover of $M$ (which only depends on $\al$) such that if $s>2,  s\notin \ZZ$ and $\beta$ is center-bunched,  then $\beta$ is   fixed point trivial if and only if  
$\beta$ lifts to a $C^{[s]-}-$coboundary.
\end{enumerate}
\end{maintheorem}

\begin{remark}Basically $r-$bunching condition (or certain domination condition, or $r-$normal hyperbolicity) is closely related to the regularity of conjugacy map between diffeomorphism or cocycles. In fact there exists examples of two $(r-)-$bunching cocycles over $\ZZ-$action are $C^{r-}-$cohomologous but not $C^{r+}-$cohomologous, cf. Theorem 5.5.3 \cite{KNbook} or \cite{dLL92}. Therefore it is reasonable to conjecture that $r-$bunching condition in part (1) of the above main results is necessary.
\end{remark}

\begin{remark} 
Part (2) in the above main results is a rigidity statement for the actions in question also in the following sense. Suppose the action has a fixed point, which is quite common.  Then the only obstruction we find is value of the cocycle at a fixed point for the action, while from Livsic theorem we know that each action element has infinitely many periodic orbit obstructions to (smooth) cocycle trivialisation. Result in part (2) of the above Theorems means that most of these obstructions for individual action elements vanish, if the cocycle is trivial at a fixed point. 
\end{remark}

\begin{remark}In part (2). or the main results above, without further assumptions for $\al$, usually it is necessary to pass to a finite cover of $M$ since there exists algebraic examples taking values in compact Lie groups (hence $\infty-$bunching) which is fixed point trivial but not cohomologous to constant. (cf. Chapter 3. \cite{NTnonabelian1})
\end{remark}



The center-bunching assumption above could be relaxed to $0-$bunching when $N=\mathbb{S}^1:=\RR/\ZZ$ since in this case $0-$bunching implies center-bunching. 
\begin{coro}
Suppose $\al, M$ satisfy the same assumptions as Theorem \ref{thm: main} and $\beta: \ZZ^k\times M\to \Diff^s(\mathbb{S}^1)$ is a $C^s-$cocycle over $\al$.\begin{enumerate}
\item  $\beta$ is essentially $C^{1}-$cohomologous to constant if $\beta$ is $0-$bunched and $s>2$.
\item There is a finite cover of $M$ such that if $s>2,  s\notin \ZZ$ and $\beta$ is $0$-bunched,  then $\beta$ is  fixed point trivial if and only if  
$\beta$ lifts to a $C^{[s]-}-$coboundary.
\end{enumerate}
As a corollary, similar results corresponding to Theorem \ref{thm: cond A},\ref{thm: Cartan} also hold for $\Diff(\mathbb{S}^1)-$valued cocycles.
\end{coro}

Outline of the paper: In Chapter \ref{prel} we give basic definitions on regularity of foliations and obtaining global regularity from regularity along transverse foliations.  In Chapter \ref{section: rig tns ano} we apply general result of Rodriguez Hertz and Wang \cite{HW} to TNS Anosov actions on infranilmanifolds. In Chapter \ref{PHextension} we obtain crucial results on the partially hyperbolic action obtained as extension of the Anosov action via a cocycle. Chapter \ref{mainproofs} contains proofs of the main results.

\section{Preliminaries}\label{prel}\color{black}
\subsection{Regularity of foliations}\label{subsection: Frob thm}
In this paper we use the notion of regularity of foliations considered by Pugh, Shub, and Wilkinson \cite{PSW}. Consider a foliation $\mathcal{W}$ of an $n-$dimensional
smooth manifold $M$ by $k-$dimensional submanifolds we define $\mathcal{W}$ to be a $C^r, r\geq 1$ foliation if for each $x\in M$ there is an open neighborhood $V_x$ of $x$ and a
$C^r$ diffeomorphism $\Psi_x : V_x\to D^k\times D^{n-k}\subset \RR^n $(where $D^j$ denotes the unit ball in $\RR^j$) such that $\Psi_x$ maps $\mathcal{W}|_{V_x}$ to the standard smooth foliation of $D^k\times D^{n-k}$ by $k-$disks $D^k\times \{y\}$, $y\in D^{n-k}$. 

In particular, by Frobenius theorem (cf. Chapter 6. of \cite{PSW} and the reference therein), if $E$ is a $C^r (r\geq 1)$ $k-$dimensional distribution on $M$ (i.e., a $C^r$ section of the Grassmannian $G^kM$), and if $E$ is involutive in the sense that it is closed under Lie brackets, then through each point $p\in M$ there passes a unique integral manifold (i.e.  an injectively immersed $k-$dimensional submanifold $V\subset M$ everywhere tangent to $F$), and together the integral manifolds $C^r$ foliate $M$. 

We will use the following classical result in the theory of partially hyperbolic systems repeatedly (cf. \cite{HPS}, \cite{CP}): if $f$ is a $C^r, r\geq 1$ partially hyperbolic diffeomorphism then the (un)stable foliations $W^{s(u)}$ have uniformly $C^r-$leaves.

To study the regularity of dynamically defined foliation, a powerful tool is to consider the associated holonomy map.  Recall that the holonomy map is defined as the following: for two transverse foliations $\mathcal{F}_1,\mathcal{F}_2$ of $M$, for two $\mathcal{F}_1-$local leaves $\mathcal{F}_1(x), \mathcal{F}_1(y)$  close enough, the local holonomy map $h^{\mathcal{F}_2}$ along $\mathcal{F}_2$ between $\mathcal{F}_1(x), \mathcal{F}_1(y)$ is defined by 
\begin{eqnarray*}
h^{\mathcal{F}_2}: \mathcal{F}_1(x)&\to &\mathcal{F}_1(y)\\
z\in \mathcal{F}_1(x)&\mapsto& \mathcal{F}_2(z)\cap \mathcal{F}_1(y)
\end{eqnarray*}
Notice that the intersection above is locally unique therefore $h^{\mathcal{F}_2}$ is locally well defined.

\begin{lemma}\label{lemma:holo reg implies foliation reg}(cf. \cite{BX}, \cite{Sadovskaya} and \cite{PSW}.)
Let $r>0, r\notin \ZZ$ be given. Suppose that $\mathcal{W}$ and $\mathcal{F}$ are two transverse foliations of $M$ such that both $\W$ and $\mathcal{F}$ have uniformly $C^{r}$ leaves. We further suppose that the local holonomy maps along $\mathcal{W}$ between any two $\mathcal{F}$-leaves are locally uniformly $C^{r}$. Then $\mathcal{W}$ is a $C^{r}$ foliation of $M$. 
\end{lemma}

\subsection{Journ\'e Lemma}
We will use the following version of Journ\'e Lemma (cf.\cite{Journe}) later.
\begin{lemma}\label{lemma Journe}
Let $M,N$ be two smooth manifolds and $\mathcal{W}_{1}, \mathcal{W}_2$ and are two continuous foliations of $M$ with uniformly $C^r-$leaves, $r\notin \ZZ$. Moreover $\mathcal{W}_1$ is uniformly transverse to $\mathcal W_2$. If a map $f:M\to N $ is uniformly $C^r$ along the leaves of the two foliations, then it is uniformly $C^r$ on M. 
\end{lemma}
\begin{proof}If $f$ is a function then Lemma \ref{lemma Journe} is proved in \cite{Journe}. For a map $M\to N$ we only need to prove the Lemma locally, so without loss of generality we assume $N$ is an open set in $\RR^n$, $f=(f_1,\cdots,f_n)$. Apply Journ\'e Lemma to each $f_i$ we get the proof.
\end{proof}

\section{Global rigidity of TNS Anosov actions on infranilmanifold}\label{section: rig tns ano}
Now we consider the action $\al$ in Theorem \ref{thm: main}. For any Anosov element of $\al$ there is a H\"older homeomorphism $h$ which conjugates it to an automorphism. Then by \cite{Wa}  $h$ also conjugates $\al$ to an action $\rho$ by affine automorphisms. The action $\rho$ is called the linearization of $\al$, which preserving the Haar measure on $M$.

 In \cite{HW} the authors proved that if $\rho$ has no rank one factor then $\al$ is smoothly conjugated to $\rho$. Here a rank-one factor of $\rho$ is a  projection of $\rho$ to a quotient infranilmanifold, which is, up to finite index, generated by a single element. The following lemma shows that TNS condition implies the no rank one condition of \cite{HW}.

\begin{lemma}For $\al$ in Theorem \ref{thm: main}, its linearisation $\rho$ has no rank-one factor.
\end{lemma}
\begin{proof}
Let $\rho$ be the linearization of  TNS  action $\al$ on an infranilmanifold.
Stable and unstable foliations of individual action elements are topologically defined, and thus they are topological invariants. Therefore the same holds true for maximal intersections of these foliations, that is, for coarse Lyapunov foliations. Suppose that  the linearisation  $\rho$ is rank-one. This would imply that for a finite index subgroup in the acting group, the action  $\rho$ is generated by a single affine map $A$. This would mean that non-trivial intersections of stable manifolds for various elements of the action are exactly either the stable manifold for $A$ or the unstable manifold for $A$. This means that for the original action, after passing to  a finite index subgroup, there are exactly two coarse Lyapunov foliations, one coinciding with the stable manifold of the topological conjugate of $A$ and the other coinciding with the unstable manifold.  It is clear that the Weyl chamber picture in this case has only one Lyapunov hyperplane, and such an action cannot be TNS.  \end{proof}
\color{black}
As a result, $\al$ is smoothly conjugated to $\rho$. Therefore, without loss of generality we may assume in the rest of the paper $\al$ is a TNS, resonance free $\ZZ^k$ action formed by affine automorphisms on an infranilmanifold. 

We denote by $E_i$ and $\chi_i$ the Lyapunov distributions and Lyapunov functionals of $\al$ respectively. Then $E_i$ is integrable for any $i$ and tangent to a smooth foliation $W_i$ in $M$. In addition there exists $C>0$ and $L>0$ such that for any $a\in \ZZ^k$ for any unit $v\in E_i$, 
\begin{eqnarray}\label{eqn: poly dev exp growth}
C^{-1}e^{\chi_i(a)}< \|D\al(a)\cdot v\|< C\|a\|^{L}e^{\chi_i(a)}
\end{eqnarray}
where $\|a\|$ is the Euclidean norm of $a$ in $\ZZ^k$.
\section{Partially hyperbolic extension}\label{PHextension}
Suppose $\beta,r,s$ are defined in (1). of Theorem \ref{thm: main}. The key point to prove (1). of Theorem \ref{thm: main} is to consider the extension action $\tilde{\alpha}$ of $\al$ on $M\times N$ induced by $\beta$:
$$\tilde{\alpha}:\ZZ^k\to \Diff^{s}(M\times N), \tilde{\al}(a)\cdot (x,y)=(\al(x),\beta(a,x)\cdot y)$$
We denote by $\pi$ the canonical projection from $M\times N$ to $M$. Then $d\pi: T(M\times N)\cong TM\oplus TN \to TM$ is the projection to its first coordinate. For Lyapunov distributions $E_i\subset TM$ of $\al$, we consider the distributions $$\bar E_i\subset TM\times TN, ~~\bar E_i:=E_i\oplus \{0\}\subset TM\oplus TN $$
In general for any $i$, $\bar E_i$ is not $\tilde{\al}-$invariant, but $\bar E_i\oplus TN$ is $\tilde{\al}-$invariant and integrable (it is tangent to the smooth $\tilde{\al}-$invariant foliation  $W_i\times N$).

The following proposition is the main result of this chapter and the main ingredient in the proof of Theorem \ref{thm: main}. It proves the properties of lifted invariant distributions and existence of the horizontal foliation, namely a foliation which is uniformly transverse to the fibers $N$. Recall that we assume $r\geq 1$ and $s>r+1$. Set $\dim E_i=d_i$.
\begin{prop}\label{prop: main prop} For any $i$, there is a $C^{r+}-$distribution $\tilde E_i\subset \bar E_i\oplus TN$ such that:
\begin{enumerate}
\item $\dim \tilde{E}_i=d_i$ and $d\pi(\tilde{E}_i)=E_i$.
\item $\tilde{E}_i$ is $\tilde\al-$invariant and tangent to a $C^{r+}-$foliaiton $\tilde{W}_i$.
\item The distribution $\oplus_i\tilde{E}_i$ is tangent to a $C^{r+}-$folation $\mathcal W_H$ with uniformly $C^{s-}-$leaves. 
\end{enumerate}

\end{prop} 

\subsection{Robustness of $\mathrm{PH}(\beta)$}We denote by $\mathrm{PH}:=\mathrm{PH}(\beta)$ the subset of Anosov elements $a$ such that the cocycle $\beta$ over $\al(a)$ satisfies \eqref{eqn: bunching condition}, i.e. $\beta$ is $0-$bunched over the $\ZZ-$action generated by $\al(a)$.
We consider the following lemma on \textit{robustness} of $\mathrm{PH}$.
\begin{lemma}\label{lemma: op con PH}For any $a\in \mathrm{PH}$,
there exists $\epsilon>0$ small enough such that for any $b\in \ZZ^k-\{0\}$ where $d(\frac{a}{\|a\|}, \frac{b}{\|b\|})<\epsilon$, we have $b\in \mathrm{PH}$.
\end{lemma} 
\begin{proof}By finiteness of Lyapunov functionals there exists $D_1>0$ such that for any $m\in \ZZ^k$, for any Lyapunov functional  $\chi$,
\begin{equation}\label{eqn: uni boun chi}
\|\chi(m)\|\leq D_1\cdot \|m\|
\end{equation}

Similarly there exists $D_2>0$ such that for any $x\in M, m\in \ZZ^k$, 
\begin{equation}\label{eqn: uni boun dbeta}
\|D\beta(m,x)\|\leq e^{D_2\|m\|}
\end{equation}  

For fixed $a\in \mathrm{PH}$, by definition there exists $k_0\geq 1, \lambda>0$ such that for any $x\in M$
\begin{equation}\label{eqn: def PH}
\|D\beta(k_0a,x)\|<e^{-\lambda}\|D_x\al(k_0a)|_{E^u_{a}}^{-1}\|^{-1}
\end{equation}

By \eqref{eqn: poly dev exp growth}, there exists $C_0>0$ for any $\chi$ such that $\chi(a)>0$, for any $n\in \ZZ^+$, we have
\begin{equation*}
\|D_x\al(nk_0a)|_{E^u_{a}}^{-1}\|^{-1}\leq C_0\cdot n^Lk_0^L \|a\|^L e^{nk_0\chi(a)}
\end{equation*}
where $L$ is defined in \eqref{eqn: poly dev exp growth}. Combine with \eqref{eqn: def PH}, for any $\chi$ such that $\chi(a)>0$, for any $n\in \ZZ^+$, we have 
\begin{equation}\label{eqn: est beta iter}
\|D\beta(nk_0a,x)\|<  e^{-n\lambda} C_0\cdot n^Lk_0^L \|a\|^L e^{nk_0\chi(a)}
\end{equation}

Notice that by subadditivity, if $\frac{b}{\|b\|}=\frac{a}{\|a\|}$ then $b\in\mathrm{PH}$. Now we pick $\epsilon $ small enough such that if $b$ satisfies $d(\frac{a}{\|a\|}, \frac{b}{\|b\|})<\epsilon$, then 
\begin{enumerate}
\item $b$ is in the same Weyl chamber as $a$.
\item There exists $N=N(\epsilon,a)$ large such that for any $\|b\|< N$, $$d(\frac{a}{\|a\|}, \frac{b}{\|b\|})<\epsilon\Rightarrow \frac{b}{\|b\|}=\frac{a}{\|a\|}$$
\end{enumerate}  
So for $b$ such that $\|b\|<N$ and $d(\frac{a}{\|a\|}, \frac{b}{\|b\|})<\epsilon$, we have $b\in \mathrm{PH}$. For $b$ with large norm,  we consider $n_0k_0a$ is the element in  $\{k_0\ZZ\cdot a \}$ which is closest to $b$. If $N$ is large enough, by geometry we have  
\begin{equation}\label{eqn: remainder}
\|b-n_0k_0a\|<2\epsilon n_0k_0\|a\|
\end{equation}
We take the Lyapunov functional $\chi_0$ such that $\chi_0(b)=\min_{\chi, \chi(b)>0}\chi(b)$. Since $b$ is in the same Weyl chamber as $a$ then $\chi_0(a)>0$. Then for any $x\in M$, by \ref{eqn: poly dev exp growth} there is $C_1>0$ does not depend on the choice of $b$ and $\chi_0$ (only depends on $\al$) such that 
\begin{eqnarray}\label{eqn: est b uns}
\|D_x\al(b)|^{-1}_{E^u_b}\|^{-1}& \geq & C_1 e^{\chi_0(b)}\\ \nonumber
&=&C_1 e^{\chi_0(n_0k_0a)}e^{\chi_0(b-n_0k_0a)}\\ \nonumber
&\geq &C_1 e^{n_0k_0\chi_0(a)}e^{-2D_1n_0k_0\epsilon\|a\|} \quad\text{     by \eqref{eqn: uni boun chi} and \eqref{eqn: remainder}}
\end{eqnarray}
On the other hand, by subadditivity
\begin{eqnarray}
\|D\beta(b,x)\| &\leq & \|D\beta(n_0k_0a,x)\|\cdot \|D\beta(b-n_0k_0a, \al(n_0k_0a)\cdot x)\|\\\nonumber 
&\leq & \|D\beta(n_0k_0a,x)\| \cdot e^{2D_2n_0k_0 \epsilon\|a\|  }\quad \text{ by \eqref{eqn: uni boun dbeta} and \eqref{eqn: remainder}}\\ \nonumber
&\leq & e^{-n_0\lambda} C_0\cdot n_0^Lk_0^L \|a\|^L e^{n_0k_0\chi_0(a)} \cdot e^{2D_2n_0k_0 \epsilon\|a\|  }\\\nonumber 
&\quad& \text{ by \eqref{eqn: est beta iter} since $\chi_0(a)>0$ }
\end{eqnarray}

Comparing with \eqref{eqn: est b uns}, we know that if we choose $\epsilon\ll \frac{\lambda}{(D_1+D_2)k_0}$, for any $x\in M$ we have $\|D\beta(b,x)\|<\|D_x\al(b)|^{-1}_{E^u_b}\|^{-1}$ which implies $b\in \mathrm{PH}$.
\end{proof}

\subsection{Existence of $\tilde{E}_i$ and $\tilde W_i$.}\label{subsection: exist ei wi}
Since for any Weyl chamber there is an elment in $\mathrm{PH}$, for any Lyapunov distribution $E_i$ by \eqref{eqn: bunching condition} we can choose an $a\in \mathrm{PH}$ (or take $na$ for $n$ large if necessary) such that
\begin{eqnarray}\label{eqn: dom split Ei TN}
\sup_{x\in M}\|D_x\al(a)|^{-1}_{E_i}\|\cdot \|D\beta(a,x)\|<1
\end{eqnarray}  
The first step to prove the existence of $\tilde{E}_i$ and $\tilde{W}_i$ is to construct a cone field on $W_i\times N$ which is contracted by a $D\tilde\al(a)$.
\begin{lemma}\label{lemma: cone E tilde }There exists $l=l(a)\in \NN$, $\gamma=\gamma(a)>0, \epsilon\in (0,1)$ such that for the cone field $$\mathcal{C}_{i, \gamma}:=\{(u,v)\in E_i\oplus TN, \|v\| \leq \gamma \cdot \|u\| \}$$ we have $$D\tilde{\al}(la)\cdot \mathcal{C}_{i,\gamma}\subset \mathcal{C}_{i,(1-\epsilon)\gamma}$$
\end{lemma}
\begin{proof}By \eqref{eqn: dom split Ei TN} there exists $\lambda_i<1$ such that $$\sup_{x\in M}\|D_x\al(a)|^{-1}_{E_i}\|\cdot \|D\beta(a,x)\|<\lambda_i$$
Consider any $(x,y)\in M\times N$ and any tangent vector $(u,v)\in E_i(x)\times T_yN$. For $n\in \NN$ we denote by $$D_{x,y}\tilde{\al}(na)\cdot (u,v)^t:=\begin{pmatrix}
A_n(x) & \\
C_n(x,y)& D_n(x,y)
\end{pmatrix}\cdot (u,v)^t$$
Then by definition of $\lambda_i$ and \eqref{eqn: poly dev exp growth}, for any $(x,y)\in M\times N$ and any $n\in \NN$, we have
\begin{eqnarray}\label{eqn An Dn}
\lambda_i^n\min_{\|u\|=1, u\in E_i(x)}\|A_n(x)\cdot u\|&>&\max_{\|v\|=1, v\in T_yN}\|D_n(x,y)\cdot v\|\\ \nonumber 
cn^{L}e^{n\chi_i(a)}\geq \max_{\|u\|=1, u\in E_i(x)}\|A_n(x)\cdot u\|&\geq& \min_{\|u\|=1, u\in E_i(x)}\|A_n(x)\cdot u\|\geq  c^{-1}e^{n\chi_i(a)}
\end{eqnarray}
where $c=c(E_i, a)\geq 1$ is a number only depends on $a$ and $E_i$. Choose an $l=l(a)$ and a $\gamma=\gamma(a)$ such that
\begin{eqnarray}
c^{-1}-\lambda^l cl^L&>&0\\
\gamma&>&\frac{\sup_{(x,y)\in M\times N}\|C_l(x,y)\|}{c^{-1}-\lambda_i^l cl^L}
\end{eqnarray}
Then we claim that there is $\epsilon\in(0,1)$ such that our choice of $(l,\gamma,\epsilon)$ satisfies Lemma \ref{lemma: cone E tilde }. In fact by \eqref{eqn An Dn}, $$\|D_l\|\leq \lambda_i^{l}\cdot cl^{L}e^{l\chi_i(a)}$$ If we denote $D_{x,y}\tilde{\al}(la)\cdot (u,v)$ by $(\tilde{u},\tilde{v})$ for any $(u,v)\in E_i(x)\oplus T_yN$ such that $\|v\|\leq \gamma\|u\|$, then we have $$\|\tilde{v}\|\leq \sup_{(x,y)\in M\times N}\|C_l(x,y)\|\cdot \|u\|+ \lambda_i^{l}\cdot cl^{L}e^{l\chi_i(a)}\cdot \gamma \|u\|$$
And by $\|A_l^{-1}\|^{-1}\geq c^{-1}e^{l\chi_i(a)}$ we know that $$\|\tilde{u}\|\geq c^{-1}e^{l\chi_i(a)}\|u\|$$ 
Therefore by our choice of $\gamma$ and $l$ we have $\frac{\|\tilde{u}\|}{\|\tilde{v}\|}>\gamma^{-1}$. Then we can easily choose $\epsilon$ satisfying the condition of Lemma \ref{lemma: cone E tilde }.
\end{proof}

As a corollary,
\begin{coro}\label{coro: smooth leaves til al}For $i$, $a$ and the cone field $\mathcal{C}_{i,\gamma}$ in Lemma \ref{lemma: cone E tilde } we have\begin{enumerate}
\item  For any $(x,y)\in M\times N$, the subset $$\tilde{E}_i(x,y):=\cap_{n\geq 0} D_{\tilde{\al}(-na)\cdot (x,y)}\tilde{\al}(na)\cdot \mathcal{C}_{i,\gamma}(\tilde{\al}(-na)\cdot (x,y))$$ is a $D\tilde{\al}(a)-$invariant $d_i-$dimensional subspace of $E_i(x)\oplus T_yN$ which continuously depends on $(x,y)$ and uniformly transverse to $TN$.
\item $\tilde{E}_i$ is invariant under $\tilde{\al}-$action for any $i$. Moreover there is $C'>0$ such that for any $b\in \ZZ^k$, any $v\in \tilde{E}_i$ with $\|v\|=1$, 
\begin{equation}\label{eqn joint ply dev exp gro}
C'^{-1}e^{\chi_i(b)}< \|D\tilde{\al}(b)\cdot v\|< C'\|b\|^{L}e^{\chi_i(b)}
\end{equation}
where $L$ is the same as in \eqref{eqn: poly dev exp growth}.
\item For any $i$, $\tilde{E}_i$ is tangent to
an $\tilde{\al}-$invariant foliation $\tilde{W}_i$ of $M\times N$ with uniformly $C^{s}-$leaves.
\end{enumerate}
\end{coro}
\begin{proof}
(1). can be proved by classical cone criterion in partially hyperbolic dynamical systems, for example see \cite{CP}.

For (2). we firstly apply Lemma \ref{lemma: cone E tilde } and (1). to any $j$ and any $a\in \mathrm{PH}(\tilde{\al})$ then we get a family of $D\tilde{\al}(a)-$invariant $d_j-$dimensional subspaces $\tilde{E}_j$ which is uniformly transverse to $TN$. Therefore by \eqref{eqn: poly dev exp growth}, we can find $C(a,i)$ such that for any $v\in \tilde{E}_j, \|v\|=1$,
\begin{equation}\nonumber
C(a,i)^{-1}e^{\chi_i(a)}< \|D\tilde{\al}(a)\cdot v\|< C(a,i)\|a\|^{L}e^{\chi_i(a)}
\end{equation}

Therefore for any $a\in \mathrm{PH}(\tilde{\al})$ which does \textbf{not} stay in any $\ker(\chi_j)\cap \ker(\chi_k)$ (by Lemma \ref{lemma: op con PH} it is possible to choose such $a$), $\tilde{E}_j$ is a Lyapunov subspace of $\tilde{\al}(a)$ in Oseledec splitting with respect to any $\tilde\al-$invariant measure $\tilde{\mu}$. Pick any $b\in \ZZ^k$, by commutavity we know $D\tilde\al(a)$ has $\chi_i(a)$ as Lyapunov exponent on $D\tilde\al(b)\tilde{E}_j$ with respect to any $\tilde\al-$invariant measure $\tilde{\mu}$. Therefore we have $D\tilde\al(b)\tilde{E}_j\subset\tilde{E}_j$ which implies $D\tilde\al(b)\tilde{E}_j=\tilde{E}_j$ for any $b,j$. So $\oplus_j \tilde{E}_j\oplus TN$ is an $\tilde{\al}-$invariant splitting. Then by transversality and \eqref{eqn: poly dev exp growth} we can find a $C'>0$ such that \eqref{eqn joint ply dev exp gro} holds.

For (3). Consider the uniformly smooth foliation $W_i\times N$ in $M\times N$. Pick any $a\in \mathrm{PH}(\tilde{\alpha})$ with $\chi_i(a)>0$, $\tilde{\al}(a)$ is \textbf{partially hyperbolic} within $W_i\times N$ (with respect to the splitting $\tilde{E}_i\oplus TN$) in the sense that there exists $k>0$ such that for any $(x,y)\in W_i\times N$, any choice of unit vectors $v \in \tilde{E}_i(x,y)$, $u\in E_i(x)\oplus T_yN$,
\[
1<\|D\tilde{\al}(a)^{k}(v)\|,~~
\|D\tilde{\al}(a)^{k}(u)\| < \|D\tilde{\al}(a)^{k}(v)\|.
\]
By the same proof of smoothness for strong (un)stable foliations  of partially hyperbolic systems, for example cf.\cite{HPS} or \cite{CP},  $\tilde{E}_i$ is tangent to an $\tilde{\al}(a)-$invariant foliation $\tilde{W}_i$ with uniformly $C^{s}-$leaves (since in the assumption of Theorem \ref{thm: main}, $\tilde{\al}(a)$ itself is $C^{s}$). By (2). we know $\tilde{W}_i$ is $\tilde{\al}-$invariant.
\end{proof}

As a corollary, for any $a\in \mathrm{PH}$, $\tilde{\al}(a)$ is a partially hyperbolic system with respect to the splitting $$\tilde{E}^s_a:=\{0\},~~ E^c:=TN\oplus\oplus_{\chi_i(a)<0} \tilde{E}_i,~~ \tilde{E}^u_a:=\oplus_{\chi_i(a)>0} \tilde{E}_i$$since by (2) of Lemma \ref{coro: smooth leaves til al} the action of $D\tilde{\al}(a)$ restricted on $\tilde{E}^{u}_a$ has the same growth speed (up to a constant) as that of $D\al(a)$ on $E^{u}_a$. So $\tilde{E}^{u}_a$ is integrable and tangent $\tilde{W}^{u}_a$ to the unstable foliation of $\tilde{\al}(a)$. By theory of partially hyperbolic systems, $\tilde{W}^{u}_a$ has uniformly $C^{s}-$leaves.

\subsection{Regularity of $\tilde{W}^i$}
In this subsection we prove that under $r-$bunching condition, $\tilde{E}_i$ is a $C^{r+}-$distribution in $M\times N$. 
Firstly we consider the following $C^r-$section theorem in \cite{HPS}.
\begin{lemma}\label{lemma: cr section}Suppose $f$ is a $C^r, r\geq 1$ diffeomorphism of a compact smooth
manifold $M$, and $W$ is an $f-$invariant topological foliation with uniformly $C^r$
leaves. Let $B$ be a normed vector bundle over $M$ and $F : B\to B$ be a linear
extension of $f$ such that both $B$ and $F$ are uniformly $C^r$ along $W$. Suppose that $F$ contracts fibers of $B$, i.e. for any $x \in M$ and any $v\in B_x$,
\begin{equation}
\|F \cdot v\|_{f(x)} ≤ k_x
\|v\|_x,~~ \sup_{x\in M}k_x< 1.
\end{equation}
Then there exists a unique continuous $F-$invariant section of $B$. Moreover, if 
$$\sup k_x \al^r_x < 1$$ where $\al_x := \|df |_{TW(x)}^{−1}\|$
then the unique invariant section is uniformly $C^r$  along the leaves of $W$.
\end{lemma}
The main idea to prove regularity of $\tilde{E}^i$ is to apply the following corollary of $C^r-$section theorem to different invariant bundles.

\begin{coro}\label{coro: coro cr section}Let $f$ be a $C^{r+1}$ diffeomorphism of a compact smooth manifold $M$. Let $W$ be an $f-$invariant topological foliation with uniformly $C^{r}-$leaves and  $\|Df |^{-1}_{TW(x)}\|:= \al_x$ for all $x\in M$. Let $E^1$ and $E^2$ be continuous f-invariant distributions on $M$ such that the distribution $E = E^1\oplus E^2$ is uniformly $C^r$ along $W$ and $E^1\oplus E^2$ is a dominated splitting in the sense that for any $x \in M$,
$$k_x:=\frac{\max_{v\in {E}^2(x), \|v\|=1}\|Df(v)\|}{\min_{v\in {E}^1(x), \|v\|=1}\|Df(v)\|}<1$$
If $\sup_{x\in M}k_x\al_x^r<1$.
Then $E^1$ is uniformly $C^r$ along the leaves of $W$. In particular if $\al_x\leq 1$ for any $x\in M$ then $E^1$ is uniformly $C^r$ along $W$.
\end{coro}

\begin{proof}(cf.\cite{KS07}) Since $E$ is $C^r$ along $W$, then we can approximate $E^1$ and $E^2$ by $\bar{E}^1$ and $\bar{E}^2$  respectively such that $\bar{E}^{1,2}$ are subbundles of $E$ and $C^r$ along $W$ and $\bar{E}^1\oplus \bar{E}^2$ is still a dominated splitting of $E$ under $df$. Moreover we can assume $$\tilde k_x:=\frac{\max_{v\in \bar{E}^2(x), \|v\|=1}\|Df(v)\|}{\min_{v\in \bar{E}^1(x), \|v\|=1}\|Df(v)\|}$$ satisfying $$\sup_{x\in M}\tilde{k}_x<1, ~~\sup_{x\in M}\tilde{k}_x\al_x^r<1$$

Define the vector bundle $B$ over $M$ where the fiber is defind by $B_x:=\{L:\bar{E}^1_x\to \bar{E}^2_x, L \text{ is linear}\}$. Then $Df$ induces a bundle map $F$ on $B$. Since $\bar{E}^1\oplus \bar{E}^2$ is a dominated splitting, $F$ contracts the fiber of $B$. In fact, by calculation we know for any $x\in B_x$, $\|F\cdot v\|_{f(x)}\leq \tilde{k}_x\|v\|_x$. Then by Lemma \ref{lemma: cr section}, there exists a unique continuous $F-$invariant section. By uniqueness, the distribution $E^1$ should be the graphs of this section. Since $\sup_{x\in M}\tilde{k}_x\cdot \al_x^r<1$. By Lemma \ref{lemma: cr section} we know $E^1$ is uniformly $C^r$ along $W$.
\end{proof}
\subsubsection{Regularity along coarse Lyapunov foliations}
Recall that for the action $\al$, $E_i$ is the Lyapunov distribution associated to the Lyapunov functional $\chi_i$, and the coarse Lyapunov distribution for $\al$ is defined as $$E^i:=E^{\chi_i}=\oplus_{\chi=c\chi_i, c>0}E_\chi=\cap_{a\in \ZZ^k-\cup_j\ker(\chi_j), \chi_i(a)>0}E^u_a.$$ 
For the action $\tilde{\al}$ we can define the coarse Lyapunov distribution similarly: for $\chi_i$, $\tilde{E}^i:=\tilde{E}^{\chi_i}=\oplus_{\chi=c\chi_i, c>0}\tilde{E}_\chi$. Since for any Weyl chamber, there is an element in $\mathrm{PH}(\tilde{\al})$, therefore $$\tilde{E}^i=\cap_{a\in \mathrm{PH}(\tilde{\al}), \chi_i(a)>0}\tilde{E}^u_a.$$
where $\tilde{E}^u, \tilde{W}^u$ are defined in the end of section \ref{subsection: exist ei wi}. Then $\tilde{E}^i$ is integrable and tangent to the coarse Lyapunov foliation $\tilde{W}^i:=\cap_{a\in \mathrm{PH}(\tilde{\al}), \chi_i(a)>0}\tilde{W}^u_a$ with uniformly $C^{s}-$leaves.

In the following proposition we prove the regulartiy of $\tilde{E}_i$ along each coarse Lyapunov foliation. Our approach generalizes of the arguments in \cite{KS07} to partially hyperbolic actions. Notice here that the quasiconformality assumptions in \cite{KS07} are not used in our proof.
\begin{prop}\label{prop: ej reg}For any $j,k$ such that $\tilde{E}_j\cap \tilde{E}^k=\{0\}$, $\tilde{E}_j$ is uniformly $C^{r+}$ along $\tilde{W}^{k}$.
\end{prop}
\begin{proof}

Consider $\tilde{E}^1=\tilde{E}_{c_1\chi_1}\oplus \tilde E_{c_2\chi_1}\oplus\cdots \oplus \tilde E_{c_l\chi_1}, ~~0<c_l<\cdots<c_1$. We take a $2-$dimensional subspace $P\subset \RR^k$ in general position such that $P$ intersects each Lyapunov hyperplane along distinct lines. In addition, since $\al$ is resonance free, $P$ can be chosen such that for any $b\in \ker \chi_1\cap P-\{0\}, \chi_i(b)\neq \chi_j(b)$ for any $(\chi_i, \chi_j)$ where $\chi_i$ is not proportional to $\chi_j$. For any $\chi_i$ we denote by $\mathcal{H}_i$ the half spaces in $\RR^k$ such that $\chi_i<0$. And $H_i:=\mathcal{H}_i\cap P$.

We now order these halfplanes counterclockwisely such that $H_1$ is the half space corresponding to $\tilde{E}^1$. Then by TNS condition there exists a unique $i > 1$ such that $$\cap_{1\leq j\leq i}H_j\cap\cap_{j'>i}-H_{j'}\neq\emptyset$$

By our setting of bunching elements, $\mathrm{PH}\cap_{1\leq j\leq i}-\mathcal{H}_j\cap\cap_{j'>i}\mathcal{H}_{j'}\neq \emptyset$ and for any element $a\in \mathrm{PH}\cap_{1\leq j\leq i}-\mathcal{H}_j\cap\cap_{j'>i}\mathcal{H}_{j'}$, by our assumption of $i$, $$\oplus_{1\leq j\leq i} \tilde{E}_j=\tilde{E}_{c_1\chi_1}\oplus \tilde E_{c_2\chi_1}\oplus\cdots \oplus \tilde E_{c_l\chi_1}\oplus \tilde E_{l+1}\oplus \cdots \oplus\tilde E_i=\tilde{E}^u_a$$ Then  $\oplus_{1\leq j\leq i} \tilde{E}_j$ is uniformly $C^{r+}$ along $\tilde{W}^u_a$  and in particular along $\tilde{W}^1$.

We choose a unit vector $b\in \ker\chi_1\cap P$ such that $b\in H_j$ for any $l+1\leq j\leq i$. By our choice of $P$ we know for any $j',j''\geq l+1$, $\chi_{j'}(b)\neq \chi_{j''}(b)$. Therefore we could reorder the indices $1,\dots, i$ by $j_i, j_{i-1},\cdots, j_1$ such that $$\chi_{j_i}(b)<\cdots <\chi_{j_{l+1}}(b)<\chi_{j_l}(b)=\cdots=\chi_{j_1}(b)=0, ~~\text{ where } \chi_{j_s}=c_{s}\chi_1, 1\leq s\leq l$$
As a result if we choose $b'\in \ZZ^k\cap -\mathcal{H}_1$ such that $\frac{b'}{\|b'\|}$ close to $b$ enough, then we have 
\begin{equation}\label{eqn: distinct chi ji}
\chi_{j_i}(b')<\cdots <\chi_{j_{l+1}}(b')<0<\chi_{j_l}(b')<\cdots<\chi_{j_1}(b')
\end{equation}

We consider an arbitrary $m$ such that $l+1\leq m< i$ and apply Corollary \ref{coro: coro cr section} to $f=\tilde{\al}(b'), W=\tilde{W}^1, E^1=\oplus_{s=1}^m \tilde{E}_{j_s}, E^2=\oplus_{s=m+1}^i\tilde{E}_{j_s}$. Notice that by \eqref{eqn joint ply dev exp gro} we have
\begin{eqnarray*}
\|D\tilde{\al}(b')|^{-1}_{TW(x)}\|=\|D\tilde{\al}(b')|^{-1}_{E^{\chi_1}(x)}\|&\leq& O(e^{-\chi_{j_l}(b')})<1 \text{ when $\|b'\|$ is large}\\
\|D\tilde{\al}(b')(v)\|&\geq& O(e^{\chi_{j_m}(b')})\text{ for any  unit vector $v\in E_1$.}\\
\|D\tilde{\al}(b')(v)\|&\leq& O(\|b'\|^L\cdot e^{\chi_{j_{m+1}}(b')})\text{ for any unit vector $v\in E_2$.}
\end{eqnarray*}
Then take $b'$ such that $\|b'\|$ large enough and \eqref{eqn: distinct chi ji} holds, by Corollary \ref{coro: coro cr section} (if necessary we could replace $b'$ by $nb'$ for $n$ large) we know $E^1=\oplus_{s=1}^m \tilde{E}_{j_s}$ is uniformly $C^{s-1}$ along the leaves of $\tilde{W}^1$ for any $m$ such that $l+1\leq m<i$.

Similarly if we take $b''\in \ZZ^k\cap -\mathcal{H}_1$ such that $\frac{b''}{\|b''\|}$ sufficiently close to $-b$, then we have, 
\begin{equation}\label{eqn: b'' distinct chi ji}
\chi_{j_i}(b'')>\cdots >\chi_{j_{l+1}}(b'')>\chi_{j_l}(b'')>\cdots>\chi_{j_1}(b'')>0
\end{equation}
We apply Corollary \ref{coro: coro cr section} to  $f=\tilde{\al}(b''), W=\tilde{W}^1, E^1=\oplus_{s=m+1}^i \tilde{E}_{j_s}, E^2=\oplus_{s=l+1}^m\tilde{E}_{j_s}$ for some $m$ such that $l+1\leq m\leq i$ then by \eqref{eqn joint ply dev exp gro}, we have
\begin{eqnarray*}
\|D\tilde{\al}(b'')|^{-1}_{TW(x)}\|=\|D\tilde{\al}(b'')|^{-1}_{E^{\chi_1}(x)}\|&\leq& O(e^{-\chi_{j_1}(b')})<1 \text{ when $\|b''\|$ is large}\\
\|D\tilde{\al}(b'')(v)\|&\geq& O(e^{\chi_{j_{m+1}}(b'')})\text{ for any unit vector $v\in E_1$.}\\
\|D\tilde{\al}(b')(v)\|&\leq& O(\|b''\|^L\cdot e^{\chi_{j_{m}}(b'')})\text{ for any unit vector $v\in E_2$.}
\end{eqnarray*}
If necessary we replace $b''$ by $nb''$ for $n$ large, we get $E^1\oplus E^2$ is a dominated splitting and $df$ is non-contracting on $W$. Then by Corollary \ref{coro: coro cr section}, $E^1=\oplus_{s=m+1}^i\tilde{E}_{j_s}$ is uniformly $C^{s-1}$ along the leaves of $\tilde{W}^1$ for any $m$ such that $l+1\leq m<i$. Therefore by taking intersection, $\tilde{E}_m$ is uniformly $C^{s-1}$ along $\tilde{W}^1$ for any $m$ such that $l+1\leq m\leq i$.

Considering the halfplanes $\{-H_l\}$, mimick the proof above we get for any $j>i$, $\tilde{E}_j$ is uniformly $C^{s-1}$ along $\tilde{W}^1$. The same proof holds for any $\tilde{W}^k$. By $s-1>r$ we get the proof.
\end{proof}
\subsubsection{Regularity along $N$}
In the following proposition we prove the regularity of $\tilde E_i$ along $N$. Notice that here is the only place where we use \eqref{eqn: bunching on center}.

\begin{prop}\label{prop: tilde ej good along n}
For any $j$, $\tilde{E}_j$ is uniformly $C^{r+}$ along $N$.
\end{prop}
\begin{proof}Notice that if $\beta$ is $r-$bunched then $\beta$ is automatically $(r+)-$bunched, therefore we only need to prove $\tilde{E}_j$ is uniformly $C^{r}$ along $N$. Since for any $j$, $\tilde{E}_j$ uniformly transverserse to $TN$, there exists $C''>1$ such that for any $j$, for any $b\in \ZZ^k$ and $(x,y)\in M\times N$  we have 
\begin{equation}\label{eqn: al control tilde}
C''^{-1}\|D_x\al(b)|_{{E}_j}^{-1}\|\leq  \|D_{(x,y)}\tilde\al(b)|_{\tilde{E}_j}^{-1}\|\leq C'' \|D_x\al(b)|_{{E}_j}^{-1}\|
\end{equation}

Since $\beta$ is a $r-$bunched cocycle, by definition for any Weyl chamber there is an element $a$ such that $\beta$ is $r-$bunched over $\al(a)$. Therefore we can choose $a$ such that there exists $k>0$, $\lambda\in(0,1)$ such that 
\begin{eqnarray}\label{eqn: al ka cont beta}
\sup_{x\in M}\|D_x\al(ka)|_{{E}_j}^{-1}\|\cdot \|D\beta(ka,x)\| &<&\lambda<1\\ \label{eqn: al ka r dominate beta}
\sup_{x\in M}\|D_x\al(ka)|_{{E}_j}^{-1}\|\cdot\| D\beta(ka,x) \|\cdot  \|D\beta(ka,x)^{-1}\|^r &<&\lambda<1
\end{eqnarray}
Therefore for $n$ large enough we have
\begin{eqnarray}\label{eqn e tilde j dominate TN}
&&\sup_{(x,y)\in M\times N}\|D_{(x,y)}\tilde\al(nka)|_{\tilde{E}_j}^{-1}\|\cdot \|D_y\beta(nka,x)\| \\ \nonumber
&\leq & C''\|D_x\al(nka)|_{{E}_j}^{-1}\| \cdot \|D\beta(nka,x)\| \quad\text{ by \eqref{eqn: al control tilde}}\\ \nonumber
&\leq & C''\lambda^n \quad \text{by \eqref{eqn: al ka cont beta} and subadditivity}\\ \nonumber
&<&1 \text{ for $n$ large}
\end{eqnarray}
Similar by \eqref{eqn: al control tilde}, \eqref{eqn: al ka r dominate beta} and subadditivity we have for $n$ large 
\begin{eqnarray}\label{eqn e tilde j r dom TN}
\sup_{(x,y)\in M\times N}\|D_{(x,y)}\tilde\al(nka)|_{{E}_j}^{-1}\|\cdot\| D_y\beta(nka,x) \|\cdot  \|D_y\beta(nka,x)^{-1}\|^r <1
\end{eqnarray}

Now we apply Corollary \ref{coro: coro cr section} to $f=\tilde{\al}(nka), W=N, E^1=\tilde{E}_j, E^2=\tilde{TN}$.  By \eqref{eqn e tilde j dominate TN} we have for any $(x,y)\in M\times N$, $k_(x,y)$ in Corollary \ref{coro: coro cr section} satisfies
\begin{eqnarray}\label{eqn contr k(x,y)}
\sup_{(x,y\in M\times N)}k_{(x,y)}&=&\sup_{(x,y)\in M\times N}\frac{\max_{v\in {E}^2(x,y), \|v\|=1}\|Df(v)\|}{\min_{v\in {E}^1(x,y), \|v\|=1}\|Df(v)\|}\\ \nonumber 
&=&\sup_{(x,y)\in M\times N}\|D_{(x,y)}\tilde\al(nka)|_{\tilde{E}_j}^{-1}\|\cdot \|D_y\beta(nka,x)\| \\ \nonumber
&<&1 \quad \text{ by \eqref{eqn e tilde j dominate TN}}
\end{eqnarray}
And for $(x,y)\in M\times N$, $\al_{(x,y)}:=\|Df |^{-1}_{TW(x)}\|=\|D_y\beta(nka,x)^{-1}\|$. Combine with \eqref{eqn contr k(x,y)}, \eqref{eqn e tilde j r dom TN} we have 
\begin{equation}\label{eqn k(x,y)al(x,y)r<1}
\sup_{(x,y)}k_{(x,y)}\al_{(x,y)}^r=\sup_{(x,y)}\|D_{(x,y)}\tilde\al(nka)|_{{E}_j}^{-1}\|\cdot\| D_y\beta(nka,x) \|\cdot  \|D_y\beta(nka,x)^{-1}\|^r <1
\end{equation}
By \eqref{eqn contr k(x,y)},\eqref{eqn k(x,y)al(x,y)r<1} and Corollary \ref{coro: coro cr section}, we get the proof of  Proposition \ref{prop: tilde ej good along n}.\end{proof}

Now we prove the coarse Lyapunov distribution $$\tilde{E}^j=\tilde{E}^{
\chi_j}=\oplus_{\chi=c\chi_j}E_\chi $$ is a $C^{r+}-$distribution of $M\times N$. The basic strategy is to apply Lemma \ref{lemma Journe} inductively. Our proof can be viewed as a partially hyperbolic version of arguments in \cite{GS},\cite{KL}. Notice that $\tilde{E}^j$ is uniformly $C^{s-1}$ (hence $C^{r+}$) along $\tilde W^j$. Then by Proposition \ref{prop: ej reg}, \ref{prop: tilde ej good along n} and transversality of distributions $\tilde{E}_\chi$ we know $\tilde E^j$ is uniformly $C^{r+}$ along any $\tilde W^k$ and $N$. As in the proof of Proposition \ref{prop: ej reg}, we consider a plane $P\subset \RR^k$ in general position such that $P$ intersects different Lyapunov hyperplanes along distinct lines. And for each \textbf{coarse} Lyapunov distribution $E^\chi$ we consider the half places $\mathcal{T}^\chi$ in $\RR^k$ such that $\chi>0$ on $\mathcal{T}^\chi$. $T^\chi:=\mathcal{T}^\chi\cap P$.

We take an arbitrary $T^\chi$ and denote by $T_1$. And then we order these halfplanes (hence the associated halfspaces) counterclockwisely. By TNS condition there exists a unique $i$ such that
$$\cap_{1\leq j\leq i}T_j\cap\cap_{j'>i}-T_{j'}\neq\emptyset$$
Therefore by our definition of $r-$bunched cocycle, we can take elements $a, a'$ satisfying $$a\in \mathrm{PH}\cap \cap_{1\leq j\leq i}\mathcal{T}_j\cap\cap_{j'>i}-\mathcal{T}_{j'}, \quad a'\in \mathrm{PH}\cap \cap_{1\leq j\leq i}-\mathcal{T}_j\cap\cap_{j'>i}\mathcal{T}_{j'}$$ then $\tilde{E}^u_a=\oplus_{1\leq j\leq i} \tilde{E}^j, \tilde{E}^u_{a'}=\oplus_{j>i}\tilde{E}^j$. By the theory of partially hyperbolic systems $\tilde{E}^u_a, \tilde{E}^u_{a'}$ are integrable and tangent to foliations $\tilde{W}^u_{a},\tilde{W}^u_{a'}$ respectively. As in the discussion at the end of section \ref{subsection: bunching}, $\tilde{W}^u_{a},\tilde{W}^u_{a'}$  have uniformly $C^{s}-$leaves. 
\begin{lemma}\label{lemma: unstabel filtr}
\begin{enumerate}\item For any $k$ ($1\leq k\leq i$), $\mathcal{E}_{k}:=\oplus_{j=k}^i\tilde{E}^j$ is integrable and tangent to a continuous foliation $\mathcal{L}_k$ with uniformly $C^{s}$ leaves. 
\item $\tilde{E}^j$ is uniformly $C^{r+}$ along $\mathcal{L}_k$ for any $k\leq i$. In particular $\tilde{E}^j$ is uniformly $C^{r+}$ along $\mathcal{L}_1=\tilde{W}^u_a$. 
\end{enumerate} 
\end{lemma}
\begin{proof}The proof of Lemma \ref{lemma: unstabel filtr} is similar to that in \cite{GS}, \cite{KL}.  For completeness we give the details here. 

For (1). by our choice of $i$ and the positions of $T_1,\dots, T_i$ we know $$\emptyset \subsetneq T_1\cap T_i\subsetneq T_2\cap T_i\subsetneq\cdots\subsetneq T_{i-1}\cap T_i\subsetneq T_i$$
In particular, $\mathcal{T}_k\cap -\mathcal{T}_{k-1}\cap \mathcal{T}_i\neq \emptyset$ $(\mathcal{T}_0=\RR^k)$ and we can choose an $a_k\in \mathrm{PH}\cap \mathcal{T}_k\cap -\mathcal{T}_{k-1}\cap \mathcal{T}_i$. By definition of $\mathcal{T}_k$ we have $\tilde{E}^u_{a_k}\cap \tilde{E}^u_a=\mathcal{E}_k$ which is tangent to a continuous foliation $\mathcal{L}_k:=\tilde{W}^u_{a_k}\cap \tilde{W}^u_a$ with uniformly $C^{s}-$leaves.

For (2), since $\mathcal{L}_i=\tilde{W}^i$, by Proposition \ref{prop: ej reg} $\tilde{E}^j$ is uniformly $C^{r+}$ along $\mathcal{L}_i$. Notice that $\tilde{W}^i$, $\tilde{W}^{i-1}$ are two uniformly transverse continuous foliation with uniformly $C^{s}-$leaves in $\mathcal{L}_{i-1}$.  Then by Proposition \ref{prop: ej reg} and Lemma \ref{lemma Journe}, $\tilde{E}^j$ is uniformly $C^{r+}$ along $\mathcal{L}_{i-1}$. 

Notice that any $k\leq i$, $\tilde{W}^{k-1}$, $\mathcal{L}_k$ are two uniformly transverse continuous foliation with uniformly $C^{s}-$leaves in $\mathcal{L}_{k-1}$. Applying Lemma \ref{lemma Journe} repeatedly, by induction we can prove that $\tilde{E}^j$ is uniformly $C^{r+}$ along $\mathcal{L}_k$ for any $k\leq i$\end{proof}
Similarly we get $\tilde{E}^j$ is uniformly $C^{r+}$ along $\tilde{W}^u_{a'}$ for any $j$. Notice that $\tilde{E}^u_a\oplus TN=\oplus_{1\leq j\leq i} \tilde{E}_j\oplus TN=\oplus_{1\leq j\leq i} {E}_j\oplus TN$ is integrable and tangent to the foliation $W^u_a\times N$ which is uniformly smooth. And $N, \tilde{W}^u_a$ are two uniformly transverse foliations within $W^u_a\times N$ and have uniformly $C^{s}-$leaves. Therefore using Lemma \ref{lemma Journe}, by (2). of Lemma \ref{lemma: unstabel filtr} and Proposition \ref{prop: tilde ej good along n} we know $\tilde{E}^j$ is uniformly $C^{r+}$ along $W^u_a\times N$. Since $\tilde{E}^j$ is also uniformly $C^{r+}$ along $\tilde{W}^u_{a'}$, and $W^u_a\times N,\tilde{W}^u_{a'} $ are two uniformly transverse foliations with uniformly $C^{s}$ leaves, apply Lemma \ref{lemma Journe} again we know $\tilde{E}^j$ is $C^{r+}$ on $M\times N$. Then $\tilde{E}_j=(E_j\oplus TN)\cap \tilde{E}^j$, as an intersection of two $C^{r+}-$distribution in $M\times N$, is a $C^{r+}-$distribution as well. Then by Frobenious Theorem  (see the discussion in section \ref{subsection: Frob thm}), $\tilde{W}_i$ is a $C^{r+}$-foliation of $M\times N$.

\subsection{The proof of Proposition \ref{prop: main prop}}\label{subsection: prf main prop}
In the previous section we proved for any $i$ the coarse Lyapunov distribution $\tilde{E}^i=\tilde{E}^{\chi_i}=\oplus_{\chi=c\chi_i,c>0}\tilde{E}_\chi$ is $C^{r+}$ on $M\times N$. By Frobenius theorem, to prove integrablity of $\oplus\tilde{E}_i=\oplus \tilde{E}^i$, we only need to prove that Lie bracket within $\oplus \tilde{E}^i$ is closed.

Suppose $X, Y$ are two $C^1$ vector fields contained in $\oplus \tilde{E}^i$ and $X=\sum_i X^i, Y=\sum_i Y^i$ are the decomposition of $X,Y$ with respect to the splitting $\oplus \tilde{E}^i$. Then $$[X,Y]=\sum_i[X^i,Y^i]+\sum_{j\neq k}[X^j,Y^k]$$
By integrability of $\tilde{E}^i$, $[X^i,Y^i]$ is contained in $E^i$ for each $i$. For $[X^j,Y^k], j\neq k$, by TNS condition there is a regular element $a\in \ZZ^k$ such that $\chi_j(a)>0, \chi_k(a)>0$. By our assumption of bunching condition there exists $b\in \mathrm{PH}$ in the same Weyl Chamber as $a$. Then by definition of $\tilde{E}^u_b$ in the end of section \ref{subsection: exist ei wi} we have $\tilde{E}^j,\tilde{E}^k\subset \tilde{E}^u_b\subset \oplus \tilde{E}^i$. Therefore by integrability of $\tilde{E}^u_b$ ($\tilde{E}^u_b=T\tilde{W}^u_b$), $[X^j,Y^k]$ is contained in $\tilde{E}^u_b$, hence in $\oplus \tilde{E}^i$. In summary,  $\oplus \tilde{E}^i$ is involutive, by Frobenius theorem $\oplus \tilde{E}^i$ is tangent to an $\tilde{\al}-$invariant $C^{r+}$ foliation $\mathcal{W}_H$.

By our assumption of bunching condition we can  choose elements $a_0, a_1\in \mathrm{PH}$ such that $-a_1$ is in the same Weyl chamber as $a_0$. Then $TM=E^u_{a_0}\oplus E^u_{a_1} $ and $W^u_{a_0}$ and $W^u_{a_1}$ are two transverse foliations of $M$ with uniformly smooth leaves, where ${E}^u_a$ and $W^u_a$ are the unstable distribution and unstable foliation of $\al(a)$ on $M$ for any regular $a\in\ZZ^k$. Since $\oplus \tilde{E}^i$ is uniformly transverse to $TN$ in $M\times N$, therefore each local leaf of $\mathcal{W}_H$ is a graph of a map $\varphi: U\subset M\to N$. Since $\mathcal{W}^u_{a_i}$ have uniformly $C^{s}-$leaves, notice that the graph of $\varphi|_{W^u_{a_i}},i=0,1$ are $\tilde{W}^u_{a_i},i=0,1$, therefore $\varphi$ is uniformly $C^s$ along $W^u_{a_i},i=0,1$. By Lemma \ref{lemma Journe} $\varphi$ is uniformly $C^{s-}$ on $U$. Therefore $\mathcal{W}_H$ has uniformly $C^{s-}$ leaves,  which implies Proposition \ref{prop: main prop}. 

\section{Proof of the main results}\label{mainproofs}
Recall that $\al$ is a TNS, resonance free $\ZZ^k$ action formed by affine automorphisms on an infranilmanifold $M$ (see chapter \ref{section: rig tns ano}). And $N$ is a smooth compact manifold. In the rest of this chapter we study the cocycle $\beta:\ZZ^k\times M\to \Diff^s(N)$ under different regularity and bunching assumptions.
\subsection{Horizontal foliation and the proof for (1). of Theorem \ref{thm: main}}\label{subsection: proof of 1.main}
In this section we assume $\beta$ is $r-$bunched and $r\geq 1, s>r+1$. We take an arbitrary point $x_0\in M$. Consider the universal covering space $(p, \hat{M}, \hat{x_0})$ of $(M, x_0)$. Then $\al$ can be lifted as an action $\hat{\al}:\ZZ^k\times \hat{M}\to \hat{M}$. And we get a lifted cocycle (as in section \ref{subsedction: def ess coho}): $$\hat{\beta}(\cdot, \cdot):=\beta(\cdot, p(\cdot)): \ZZ^k\times \hat{M}\to \Diff^{s}(N)$$ 

We claim that $\hat\beta$ is $C^{r+}-$cohomologous to a constant cocycle, which implies (1). of Theorem \ref{thm: main}. The map $p$ induces a covering map: $$\mathrm{Pr}: \hat{M}
\times N\to M\times N,~~ \mathrm{Pr}(x,y)=(p(x),y) $$
We denote by $\hat{\pi}$ the canonical projection from $\hat{M}\times N$ to $M$. Then $\tilde{\al}$ can be lifted to an action $\hat{\tilde\al}$, $$\hat{\tilde\al}: \ZZ^k\times \hat{M}
\times N\to \hat{M}
\times N, \hat{\tilde\al}(a)(x,y)=(\hat\al(a)x, \beta(a, p(x))\cdot y)$$

The $\tilde\al-$invariant foliation $\mathcal{W}_H$ defined in section \ref{subsection: prf main prop} can be lifted as an $\hat{\tilde\al}-$invariant uniformly $C^{r+}-$foliation $\hat{\mathcal{W}}_H$ of $\hat{M}\times N$.  Moreover $\hat{\mathcal{W}}_H$ is \textit{horizontal} in the sense that $\hat{\mathcal{W}}_H$ is uniformly transverse to the fiber $N$ of the fiber bundle $\hat{\pi}: M\times N\to M$ (cf. \cite{NTnonabelian1}).

By theory of suspension  in foliation theory (cf. pp. 124, section 1.2 of \cite{HH} or \cite{NTnonabelian1}), we have that the foliation $\hat{\mathcal{W}}_H$ is a uniformly $C^{r+}$ global section of the fiber bundle $\hat{M}\times N$ in the sense that each leaf of $\hat{\mathcal{W}}_H$ intersects each fiber $N$ at exactly one point.

As a corollary, we can define a $C^{r+}$ map $h$ which is induced by the holonomy of $\hat{\mathcal{W}}_H$ in $\hat{M}\times N$: 
$$h: \hat{M}\to \Diff^{r+}(N), ~~ h(\hat x)\cdot y:=\pi_N(\hat{\mathcal{W}}_H(\hat{x_0}, y)\cap N_{\hat x})$$where $\pi_N$ is the canonical projection to $N$. Since $\hat{\mathcal{W}}_H$ is a global section, $h$ is well-defined. Moreover by $\hat{\tilde\al}-$invariance of $\hat{\mathcal{W}}_H$, for any $a\in \ZZ^k$, $\hat{x}\in \hat{M}$, we have $$h(\hat{{\al}}(a)\cdot \hat x)^{-1}\circ \hat\beta(a,\hat{x})\circ h(\hat x)=h(\al(a)\cdot \hat{x_0})^{-1}\cdot \hat{\beta}(a,\hat{x_0})$$which does not depend on $\hat{x}$. Therefore $\hat{\beta}$ is $C^{r+}-$cohomologous to a constant cocycle.
\subsection{Proof for (2). of Theorem \ref{thm: main}.}\label{subsection: proof 2. main}
In this section we assume and $\beta$ is a center-bunched $C^s-$cocycle ($s\notin \ZZ$ and $s>2$) over $\al$. Our plan is firstly to prove Proposition \ref{prop: former 3.2} below  and then deduce (2). of Theorem \ref{thm: main} in section \ref{subsection: 3.2 from 6.1}.
\begin{prop}\label{prop: former 3.2}There is a finite cover $M^\ast$ of $M$ only depends on $\al$ such that if $\beta$ is trivial at a fixed point of $\al$, then $\beta$ lifts to a $C^{[s]-}-$coboundary on the cover.
\end{prop}

Notice that $\beta$ satisfies all conditions  in (1). of Theorem \ref{thm: main} for the case $s=s, r=1$. In particular without loss of generality we assume the base point $x_0$ of $M$ in section \ref{subsection: proof of 1.main} is the fixed point for $\al$. Therefore all the results and concepts in Chapter \ref{PHextension} and section \ref{subsection: proof of 1.main} could be applied to $\beta$, for example $\hat{M}\times N, \hat{\mathcal{W}}_H, \mathcal{W}_H, h, E_i, \tilde{E}_i $, etc.


\subsubsection{Construction of a finite cover}\label{subsection: M ast def}
We plan to construct a finite cover $M^\ast$ for $M$ which satisfies conditions in Proposition \ref{prop: former 3.2}. For any $\omega\in \pi_1(M, x_0)$, we consider the desk transformation  induced by $\omega$ on $\hat{M}$, $\hat{x}\in \hat{M}\mapsto \omega\cdot \hat{x}$. The following lemma is a basic property for the lift of horizontal foliations.
\begin{lemma}\label{lemma: repre H def}The diffeomorphism $ h(\hat{x})^{-1}\circ h(\omega\cdot \hat{x})$ does not depend on the choice of $\hat{x}\in \hat{M}$. In particular, it induces a well-defined group homomorphism $H:\pi_1(M, x_0)\to \Diff(M)$.
\end{lemma}
\begin{proof}
We only need to prove that for any $\omega\in \pi_1(M,x_0)$, for any  $\hat{x},\hat{y}\in \hat{M}$ which are close enough, $$h(\omega\cdot \hat{y})h(\omega\cdot\hat{x})^{-1}=h(\hat{y})h(\hat{x})^{-1}$$

By definition of $h$ we know $h(\omega\cdot \hat{y})h(\omega\cdot\hat{x})^{-1}$ is the holonomy map along $\hat{\mathcal{W}}_H$ between $N_{\omega\cdot \hat{x}}$ and $N_{\omega\cdot \hat{y}}$ and $h(\hat{y})h(\hat{x})^{-1}$ is the holonomy map along $\hat{\mathcal{W}}_H$ between $N_{\hat{x}}$ and $N_{\hat{y}}$. Notice that $\hat{\mathcal{W}}_H$ has exactly the same geometry around $\hat{x}$ and $\omega\cdot \hat{x}$ (locally they are two identical copies of $\mathcal{W}_H$ near $x$). Therefore we have $h(\omega\cdot \hat{y})h(\omega\cdot\hat{x})^{-1}=h(\hat{y})h(\hat{x})^{-1}$.\end{proof}

Consider the homomorphism $H$ defined in Lemma \ref{lemma: repre H def}, we have 
\begin{lemma}\label{lemma: finite image H}
\begin{enumerate}
\item For any $a\in \ZZ^k, \omega\in \pi_1(M, x_0)$, $H(\al(a)_\ast \omega)=H(\omega)$. Therefore $H$ induces a group homomorphism $$\bar{H}: \pi_1(M, x_0)/\mathrm{span}\{\al(a)_\ast\omega\cdot \omega^{-1}| a\in \ZZ^k, \omega\in \pi_1(M,x_0)\}\to \Diff(N)$$
\item $\mathrm{span}\{\al(a)_\ast\omega\cdot \omega^{-1}| a\in \ZZ^k, \omega\in \pi_1(M,x_0)\}$ is  \textbf{finite} index subgroup of $\pi_1(M, x_0)$.
\end{enumerate}
\end{lemma}
\begin{proof}
(1): By definition of $\hat{\beta}$ and $\hat{\tilde\al}$ we have for any $\hat{x}\in \hat{M}, a\in \ZZ^k, y\in N$
\begin{equation}\label{eqn: appdx proof 6.2}
\hat\beta(a, \hat{x})\cdot y=\pi_N(\hat{\tilde\al}(a)(\hat{x},y))
\end{equation}
Then we have for any $y\in N, a\in \ZZ^k, \omega\in \pi_1(M,x_0)$,
\begin{eqnarray*}
H(\omega)\cdot y&=& h(\omega\cdot \hat{x_0})\cdot y ~~(\text{since $ h(\hat{x_0})=id$})\\
&=&\pi_N(\hat{\mathcal{W}}_H(x_0,y)\cap N_{\omega\cdot \hat{x_0}})\\
&=&\hat{\beta}(a, \omega\cdot \hat{x_0})\cdot \pi_N(\hat{\mathcal{W}}_H(\hat{x_0},y)\cap N_{\omega\cdot \hat{x_0}})~~(\text{since $\beta(a, x_0)=id$})\\
&=&\pi_N(\hat{\tilde\al}(a)(\omega\cdot \hat{x_0},~~ \pi_N(\hat{\mathcal{W}}_H(\hat{x_0},y)\cap N_{\omega\cdot \hat{x_0}}) ))~~(\text{by \eqref{eqn: appdx proof 6.2}})\\
&=&\pi_N(\hat{\tilde\al}(a)(\hat{\mathcal{W}}_H(\hat{x_0},y)\cap N_{\omega\cdot \hat{x_0}}))\\
&&(\text{ since the first coordinate of }\hat{\mathcal{W}}_H(x_0,y)\cap N_{\omega\cdot \hat{x_0}})~~\text{is }\omega\cdot \hat{x_0})\\
&=&\pi_N(\hat{\mathcal{W}}_H(\hat{x_0},y)\cap N_{(\al(a)_\ast\omega)\cdot \hat{x_0}}))~~(\text{since }\beta(a, x_0)=id)\\
&=&H(\al(a)_\ast \omega)\cdot y
\end{eqnarray*}
Therefore $H$ induces a group homomorphism $$\bar{H}: \pi_1(M, x_0)/\mathrm{span}\{\al(a)_\ast\omega\cdot \omega^{-1}| a\in \ZZ^k, \omega\in \pi_1(M,x_0)\}\to \Diff(N)$$\\
(2): (See also \cite{NTnonabelian1}) Since $M$ is an infranilmanifold, $\pi_1(M, x_0)$ is the extension of a nilpotent group $\Lambda$ by a finite group $F$, where $\Lambda$ is a discrete subgroup in a connected, simply connected nilpotent Lie group $\mathcal{N}$. Therefore it is easy to see that we only need to prove the case when $M$ is a \textbf{nilmanifold} $\mathcal{N}/\Lambda$. 

Firstly we consider the case $\mathcal{N}$ is Abelian, then $\Lambda\cong\ZZ^l$. We consider an Anosov element $a_0\in \ZZ^k$, then the homomorphism $\al(a_0)_\ast$ induced by $\al(a_0)$ on the fundamental group of $M$ is the
restriction of an automorphism $\bar\al (a_0)$ of $\mathcal{N}$ that preserves $\Lambda$ ($\bar\al$ can be seen as the linear part of $\al$). And $D\bar\al (a_0)$ at the origin has no eigenvalues on the unit circle. By condition we know $\mathrm{span}\{\al(a)_\ast\omega\cdot \omega^{-1}| a\in \ZZ^k, \omega\in \pi_1(M,x_0)\}$ contains $(\bar\al(a_0)-id)\cdot \Lambda$. Since $\bar\al(a_0)-id$ is invertible on $\Lambda\otimes \QQ$, then $\mathrm{span}\{\al(a)_\ast\omega\cdot \omega^{-1}| a\in \ZZ^k, \omega\in \pi_1(M,x_0)\}$ is of finite index in $\Lambda$, therefore we prove the claim when $\mathcal{N}$ is Abelian.

For general $\mathcal{N}$, we need the following facts stated in \cite{Ma} and \cite{Mal}:
\begin{lemma}\label{lemma: MM}Let $f$ be an Anosov diffeomorphism on a compact nilmanifold $M=\mathcal{N}/\Lambda$. Suppose the upper central series of $\Lambda$ is $\{e\}=\Lambda_0\subset \Lambda_1\subset\cdots \subset \Lambda$, then \begin{enumerate}
\item the automorphism $f_\ast$ induced by $f$ on $\Lambda$  preserves $\Lambda_i$.
\item $M$ is expressed as a sequence of extensions by tori whose fundamental groups is \textbf{free} Abelian group $\Lambda_i/\Lambda_{i-1}$.
\item If we denote by $\varphi_i: \Lambda_i/\Lambda_{i-1}\to \Lambda_i/\Lambda_{i-1}$ the automorphism induced by $f_\ast$. Then none of the $\varphi_i$ have a root of unity as an eigenvalue.
\end{enumerate} 
\end{lemma}
As before we take an Anosov element $a_0$. Denote by $K=K(\Lambda, a_0):=\mathrm{span}\{\al(a_0)_\ast\omega\cdot \omega^{-1}|  \omega\in \Lambda_1\}$ and $Q$ the projection $\Lambda\to \Lambda/K$, we only need to prove that $\#\mathrm{Image}(Q)<\infty$. By our arguments above, $Q|_{\Lambda_1}$ has finite image.

Now we consider $\Lambda_2$ and the cosets $\{\omega\Lambda_1, \omega\in \Lambda_2\}$ of $\Lambda_1$ in $\Lambda_2$. Notice that by $Q$'s definition  
\begin{equation}\label{eqn: a action on quot L2L1}
\mathrm{Image}(Q|_{\omega_1\Lambda_1})=\mathrm{Image}(Q|_{\omega_2\Lambda_1})~~ \text{if }\al(a_0)_\ast\omega_1=\omega_2,~~ \omega_1,\omega_2\in \Lambda_2
\end{equation}
We denote by $\al(a_0)|_{2,1}$ the induced action of $\bar{\al}(a_0)$ defined above on $\Lambda_2/\Lambda_1$, since $\Lambda_2/\Lambda_1$ is a free Abelian group, we could define $(\bar\al(a_0)|_{2,1}-id)\cdot (\Lambda_2/\Lambda_1)$ which is a subgroup in $\Lambda_2/\Lambda_1$. Moreover by (3). of Lemma \ref{lemma: MM} we know $\bar\al(a_0)|_{2,1}-id$ is invertible on $(\Lambda_2/\Lambda_1)\otimes \QQ$. Therefore 
\begin{equation}\label{eqn: a action induce fin ind}
(\bar\al(a_0)|_{2,1}-id)\cdot \Lambda_2/\Lambda_1 ~~\text{ has finite index in $\Lambda_2/\Lambda_1$}
\end{equation}
Denote by $(\bar\al(a_0)|_{2,1}-id)\cdot \Lambda_2=
\{ \omega'| \omega'\in \omega \Lambda_1,~~ \omega\Lambda_1\in (\bar\al(a_0)|_{2,1}-id)\cdot \Lambda_2/\Lambda_1  \}$ then by \eqref{eqn: a action on quot L2L1} we have $$\mathrm{Image}(Q|_{(\bar\al(a_0)|_{2,1}-id)\cdot \Lambda_2})=\mathrm{Image}(Q|_{\Lambda_1})$$
Combine with \eqref{eqn: a action induce fin ind} since $Q|_{\Lambda_1}$ has finite image we know $Q|_{\Lambda_2}$ has finite image as well. Repeat the arguments above by induction  we can prove $Q|_{\Lambda}$ has  finite image.
\end{proof}
In particular, we choose $(M^\ast, x_0^\ast)$ as a finite cover of $(M,x_0)$ corresponding to the subgroup $\mathrm{span}\{\al(a)_\ast\omega\cdot \omega^{-1}| a\in \ZZ^k, \omega\in \pi_1(M,x_0)\}$ in $\pi_1(M, x_0)$. Then for any $\omega\in \pi_1(M^\ast, x_0^\ast)\hookrightarrow \pi_1(M, x_0)$ we have $h(\omega\cdot \hat{x})=h(\hat{x}), \hat{x}\in \hat{M}$. Therefore the lift $\mathcal{W}_H^\ast$ of $\mathcal{W}_H$ on $M^\ast$ is a $C^{1+}-$global section of the fiber bundle $M^\ast \times  N$. And $h$ induces a well-defined $C^{1+}$ map: 
\begin{equation}\label{eqn h star def}
h^\ast: {M}^\ast\to \Diff^{1+}(N), ~~ h^\ast (x^\ast)\cdot y:=\pi_N({\mathcal{W}}_H^\ast(x_0^\ast, y)\cap N_{x^\ast})
\end{equation}
Therefore the lifted cocycle $\beta^\ast$ on $M^\ast$ satisfies 
\begin{equation}\label{eqn: cob beta star on M star}
h^\ast({{\al}^\ast}(a)\cdot x^\ast)^{-1}\circ \beta^\ast(a,x^\ast)\circ h^\ast( x^\ast)=\beta^\ast(a,x_0^\ast)=id, ~~a
\in \ZZ^k
\end{equation}where ${{\al}^\ast}$ is the lift of ${\al}$. As a result $\beta^\ast$ is a $C^{1+}-$coboundary on $M^\ast$.
\subsubsection{Dependence on parameters of the solutions of cohomology equations}\label{subsection: improve regl}
We claim that the lifted cocycle $\beta^\ast$ on $M^\ast$ we got in section \ref{subsection: M ast def} is actually a $C^{[s]-}-$coboundary. The main idea is to use Proposition \ref{prop: DLLAW} below which is a special case of \cite{LW}.

Before state it, we define a new regularity class of maps from $M$ to $\Diff^n(N), n\in \ZZ^+$:  \textbf{H\"older} maps from $M$ to $\Diff^n(N)$. Our definition here is similar to that in \cite{LW}. Suppose $U\subset \RR^n$ is an open set. A map $h:M\to C^n(U), n\in \ZZ^+$ is called H\"older if there exists $\epsilon>0$ such that for any $y\in U$, any $1\leq i\leq n$, $D_y^ih(x)$ depends on $x$ uniformly H\"older continuously with exponent $\epsilon$ and the H\"older constant does not depend on $y$. By taking finite bounded charts for $N$, we can easily define H\"olderness of a map $h: M\to \Diff^n(N)$. 

Notice that under this definition of H\"olderness a $C^n-$map (in the sense of section \ref{subsection regl}) $h: M \to \Diff^n(N)$ may \textbf{not} be H\"older since the $n-$th derivative of $h(x)\in \Diff^n(N)$ may not  depend on $x$ H\"older continuously. But a $C^s (s\notin \ZZ, s>1)$ map from $M$ to $\Diff^s(N)$ is a H\"older map from $M$ to $\Diff^{[s]}(N)$.
\begin{prop}\label{prop: DLLAW}Let $M,N$ be smooth compact manifolds and $f$ be a smooth transitive Anosov diffeomorphism on $M$. If there are H\"older maps $\eta: M\to \Diff^n(N), n\in \ZZ^+,  \varphi\to \Diff^1(N)$ such that $\eta(x)=\varphi(f(x))\circ\varphi(x)^{-1}$. Then in fact $\varphi$ is a H\"older map from $ M$ to $\Diff^n(N)$.
\end{prop}

We come back to the proof of Proposition \ref{prop: former 3.2}. Take an arbitrary regular element $a\in \ZZ^k$, then the lift $\al^\ast(a)$ of $\al(a)$ on $M^\ast$ is a smooth transitive Anosov diffeomorphism. Since $s\notin \ZZ$, then $\beta^\ast(a, \cdot)$ is actually a H\"older map from $M^\ast$ to $\Diff^{[s]}(N)$. Similarly since $h^\ast: M^\ast\to \Diff^{1+}(N)$ is a ${C}^{1+}$ map, then $h^{\ast}$ is also a H\"older map from $M^\ast\to \Diff^1(N)$. We apply Proposition \ref{prop: DLLAW} with $M=M^\ast, N=N, f=\al^\ast(a), \eta=\beta^\ast(a, \cdot), \varphi=h^\ast, n=[s]$, by \eqref{eqn: cob beta star on M star} we have that $h^\ast$ is a H\"older map from $M$ to $\Diff^{[s]}(N)$.

As a result, by $h^\ast$'s definition in \eqref{eqn h star def} we have that the local holonomy map along $\mathcal{W}^\ast_H$ between two $N$ leaves in $M^\ast\times N$ is uniformly $C^{[s]}$. By (3). of Proposition \ref{prop: main prop}, $\mathcal{W}^\ast_H$ has uniformly $C^{s-}-$leaves, therefore by Lemma \ref{lemma:holo reg implies foliation reg}, $\mathcal{W}^\ast_H$ is a $C^{[s]-}-$foliation of $M^\ast \times N$, hence $h^\ast$ is a $C^{[s]-}-$map from $M^\ast$ to $\Diff^s(N)$. Then $\beta^\ast$ is a $C^{[s]-}-$coboundary.
\subsubsection{Proof of (2). of Theorem \ref{thm: main} by Proposition \ref{prop: former 3.2}}\label{subsection: 3.2 from 6.1}
Now we prove (2). of Theorem \ref{thm: main} without assuming the existence of $\al-$fixed point. Recall that $\beta$ is a center-bunched $C^s-$cocycle ($s\notin \ZZ, s>2$) over $\al$ where $\al$ and $M$ are defined as in the beginning of this chapter. Then (1). of Theorem \ref{thm: main} can be applied to $\beta$. 

By Remark \ref{rema: comm fix pnt} and footnote therein we can find a free Abelian subgroup $A$ of $\ZZ^k$  such that $\al|_A$ has a fixed point. Therefore $\al|_A$ satisfies all our assumptions for $\al$ in Theorem \ref{thm: main}. We apply (1). of Theorem \ref{thm: main} for the case $s=s, r=1$ to $\beta$, where we choose the base point $x_0$ of $M$ in section \ref{subsection: proof of 1.main} to be the fixed point of $\al|_A$. Similar to section \ref{subsection: M ast def} we choose the finite cover $(M^\ast, x^\ast_0)$ of $(M,x_0)$ to be the cover corresponding to the subgroup $\mathrm{span}\{\al(a)_\ast\omega\cdot \omega^{-1}| a\in A, \omega\in \pi_1(M,x_0)\}$ in $\pi_1(M, x_0)$. Notice that in fact in Lemma \ref{lemma: finite image H} we proved that for any regular element $a_0$, the group $\mathrm{span}\{\al(a_0)_\ast\omega\cdot \omega^{-1}, ~~\omega\in \pi_1(M,x_0)\}$ is  a finite index subgroup of $\pi_1(M,x_0)$. Therefore $M^\ast$ is a finite cover of $M$.

Suppose now $\beta$ lifts to a $C^{[s]-}-$coboundary $\beta^\ast$ on $M^\ast$, i.e. there exists a $C^{[s]-}-$map $h^\ast$ from $M^\ast$ to $\Diff^{[s]-}(N)$ such that for any $a\in \ZZ^k,x^\ast\in M^\ast$, 
\begin{equation}\label{eqn: lift cobound}
h^\ast(\al^\ast(a)\cdot x^\ast)^{-1}\beta^\ast(a,x^\ast)h^\ast(x^\ast)=id
\end{equation}
 where $\al^\ast, \beta^\ast$ are the lifts of $\al$ and $\beta$ on $M^\ast$ respectively. We can easily find a set $S$ of generators of $\ZZ^k$ such that all elements are regular. Then for any $a\in S$, $\al^\ast(a)$ has a fixed point $x^\ast_a$ on $M^\ast$ (cf. Remark \ref{rema: ano dif has fix pnt} and reference therein). Apply equation \eqref{eqn: lift cobound} with $a=a, x^\ast=x^\ast_a$ we know $\beta^\ast(a,x^\ast_a)=id$. Therefore $\beta$ is fixed point trivial in the sense of section \ref{subsedction: def ess coho}.
 
Conversely suppose $\beta$ is fixed point trivial, i.e. there is a set $S$ of generators for $\ZZ^k$ such that for any $a\in S$, there is an $\al(a)-$fixed point $x_a$ which satisfies $\beta(a,x_a)=id$. Notice that by (1). of Theorem \ref{thm: main} we know there is a $C^{1+}$ map $h:\hat{M}\to \Diff^{1+}(N)$ such that 
\begin{equation}\label{eqn: repeat thm3.1}
h(\hat\al(a)\cdot\hat{x})\circ \hat{\beta}(a,\hat{x})\circ h(\hat{x})^{-1}=\beta_0(a)
\end{equation}
where $\beta_0:\ZZ^k\to \Diff^{1+}(N)$ is a constant cocycle (hence a group homomorphism). Apply \eqref{eqn: repeat thm3.1} to the case $a\in S$ and $\hat{x}:=\hat{x_a}$ (the lift of $\hat{x_a}$ on $\hat{M}$) we know for any $a\in S$, $\beta_0(a)=id$. Therefore $\beta_0$ is trivial and $\beta(a,x_0)=id $ for any $a\in A$. Then by our choice of $M^\ast$ we know for any $\omega\in \pi_1(M^\ast,\hat{x_0})\hookrightarrow \pi_1(M,x_0)$ we have $h(\omega\cdot \hat{x})=h(\hat{x})$. As the end of section \ref{subsection: M ast def} we know $h$ induces a well-defined map $h^\ast$ on $M^\ast$. Then $\beta$ lifts to a $C^{1+}-$coboundary $\beta^\ast$ on $M^\ast$. By discussion in section \ref{subsection: improve regl}, $\beta^\ast$ is in fact a $C^{[s]-}-$coboundary.

\subsection{Proof of Theorems \ref{thm: cond A} and \ref{thm: Cartan}}Let $s=\infty$ we only need the following Lemma.

\begin{lemma}\label{properties} a)  $\al$ is maximal then $\al$ is full.

b)   If $\al$ is full, then  $\al$ is TNS and resonance free, with respect to any invariant ergodic measure.
\end{lemma}
\begin{proof}
If a $\mathbb Z^k-$action $\al$ is maximal, then it has exactly $k+1$ Lyapunov hyperplanes in general position. This implies that  obviously there must be at least two Lyapunov hyperspaces, and  that there are exactly $2^{k+1}-2$ Weyl chambers. Since there is no Weyl chamber where all Lyapunov exponents are positive (or all negative), it follows that all combinations of signs appear among Weyl chambers, so for any  $i$ there is Weyl chamber in which $\chi_i$ is positive while all other non-positively proportional Lyapunov functionals are negative. This implies the action is full.

To prove part b): suppose $\al$ is not TNS and that there are $i, j$ such that $\chi_i=c\chi_j$ for some $c<0$. Then these two Lyapunov functionals share the same Weyl chamber wall i.e. $\ker \chi_i=\ker \chi_j$. Since $\al$ is assumed to be not rank-one, there is at least one more Lyapunov exponent $\chi_k$ which is not proportional to   $\chi_i$ and $\chi_j$. Since $\al$ is full there exists a regular element $a$ such that $E_k=E^u_a$. This implies $\chi_k(a)>0$, but $\chi_i(a)<0$ and $\chi_j(a)<0$.  The last two inequalities are not possible for any regular element because  $\chi_i$ and $\chi_j$ are negatively proportional. 

Suppose that  $\al$  is not resonance free. Then there are three Lyapunov functionals $\chi_i, \chi_j$ and $\chi_k$ such that $\chi_i-\chi_j = c\chi_k$ for some $c\ne 0$. From assumption (A) there exists regular element $a$  for which  $E_j=E^s_a$. Then $\chi_j(a)<0$, but $\chi_i(a)>0$ and $\chi_k(a)>0$. This implies $c>0$. By the same reasoning, there exists regular element $b$ for which $E_i=E^u_b$. Then $\chi_i(b)>0$, but $\chi_j(b)<0$ and $\chi_k(b)<0$. This implies $c<0$. Therefore we can conclude $c=0$ which contradicts the assumption.\end{proof}

If $\al$ is maximal Cartan action on $M$ with all elements Anosov and at least one element transitive, then the main result in \cite[Corollary 1.4]{KSp}  shows that $\al$ is smoothly conjugate to an action on a infranilmanifold, by affine maps. Therefore, by the lemma above, we have that $\al$ (after a smooth conjugacy) satisfies the conditions of Theorem \ref{thm: main}, so we get the conclusion of Theorem \ref{thm: Cartan}.

Similarily, if $\al$ is full, and on infranilmanifold, then by the lemma above the conditions of Theorem \ref{thm: main} are satisfied, so Theorem \ref{thm: cond A} follows.

\end{document}